\documentclass[a4paper,11 pt]{article}

\usepackage{todonotes}
\usepackage{amssymb}
\usepackage{mathrsfs}
\usepackage{dsfont}
\usepackage{amsfonts}
\usepackage{enumerate, xspace}
\usepackage{color}
\usepackage{changebar}
\usepackage{pdfsync}
\usepackage{graphicx}
\usepackage{a4wide}
\usepackage{amsmath,amsthm,amstext,amscd,amssymb,euscript,mathrsfs}
\usepackage{calrsfs}
\usepackage{epsfig}
\usepackage{verbatim}

\usepackage{color}
\usepackage{esint}
\usepackage{pdfsync}

\usepackage{todonotes}

\newcommand{\D}{\mathcal D}
\newcommand{\s}{\mathcal S}

\newcommand{\Z}{\mathbb Z}
\newcommand{\R}{\mathbb R}
\newcommand{\N}{\mathbb N}

\newcommand{\E}{\mathbb E}

\renewcommand{\phi}{\varphi}

\newcommand{\la}{\ensuremath{\Lambda}}
\newcommand{\si}{\ensuremath{\sigma}}

\newcommand{\rphis}[5]{\,_{#1}\varphi_{#2}\!\left( \genfrac{.}{.}{0pt}{}{#3}{#4}
\,;#5 \right)}
\newcommand{\A}{ {A}}

\def\1{{\mathchoice {\rm 1\mskip-4mu l} {\rm 1\mskip-4mu l}
{\rm 1\mskip-4.5mu l} {\rm 1\mskip-5mu l}}}

\newtheorem{theorem}{{\small T}{\scriptsize HEOREM}}[section]
\newtheorem{definition}[theorem]{{\bf{\small D}{\scriptsize EFINITION}}}
\newtheorem{corollary}[theorem]{{\bf{\small C}{\scriptsize OROLLARY}}}
\newtheorem{proposition}[theorem]{{\bf{\small P}{\scriptsize ROPOSITION}}}
\newtheorem{lemma}[theorem]{{\bf{\small L}{\scriptsize EMMA}}}
\newtheorem{remark}[theorem]{{\bf{\small R}{\scriptsize EMARK}}}

\renewenvironment{proof}[1]
{\noindent{{\bf{\small{P}{\scriptsize ROOF}}}.}\hspace{0.1cm} #1} {$\;\qed$\newline}

\newcommand{\beq}{\begin{eqnarray}}
\newcommand{\eeq}{\end{eqnarray}}

\newcommand{\ba}{\begin{align*}}
\newcommand{\ea}{\end{align*}}

\newcommand{\be}{\begin{equation}}
\newcommand{\ee}{\end{equation}}

\newcommand{\bl}{\begin{lemma}}
\newcommand{\el}{\end{lemma}}

\newcommand{\br}{\begin{remark}}
\newcommand{\er}{\end{remark}}

\newcommand{\bt}{\begin{theorem}}
\newcommand{\et}{\end{theorem}}

\newcommand{\bd}{\begin{definition}}
\newcommand{\ed}{\end{definition}}

\newcommand{\bp}{\begin{proposition}}
\newcommand{\ep}{\end{proposition}}

\newcommand{\bc}{\begin{corollary}}
\newcommand{\ec}{\end{corollary}}

\newcommand{\bpr}{\begin{proof}}
\newcommand{\epr}{\end{proof}}

\newcommand{\bi}{\begin{itemize}}
\newcommand{\ei}{\end{itemize}}

\newcommand{\ben}{\begin{enumerate}}
\newcommand{\een}{\end{enumerate}}

\newcommand{\caL}{{\mathcal L}}

\renewcommand{\(}{\left(}        \renewcommand{\)}{\right)}

     \newcommand{\nn}{\nonumber}

\title{$q-$Orthogonal dualities for asymmetric particle systems}

\author{Gioia Carinci\footnote{University of Modena and Reggio Emilia, FIM, via G. Campi 213/b, 41125 Modena, Italy.} \quad Chiara Franceschini\footnote{Instituto Superior T\'ecnico, CAMGSD, Av. Rovisco Pais, 1049-001 Lisboa, Portugal.}\quad
Wolter Groenevelt\footnote{Technische Universiteit Delft, DIAM, P.O. Box 5031, 2600 GA Delft, The Netherlands.}
}

\date{\today}

\begin{document}

\maketitle
\abstract{We study a class of interacting particle systems with asymmetric interaction showing a self-duality property. The class includes the ASEP($q,\theta$), asymmetric exclusion process, with a repulsive interaction, allowing up to $\theta\in \N$  particles in each site, and the ASIP$(q,\theta)$, $\theta\in \R^+$, asymmetric inclusion process, that is its attractive counterpart.  We extend to the asymmetric setting the investigation of orthogonal duality properties done in \cite{CFGGR} for symmetric processes. The analysis leads to {multivariate} $q-$analogues of {Krawtchouk} polynomials and Meixner polynomials as orthogonal duality functions for the generalized asymmetric exclusion process and its asymmetric inclusion version, respectively. We also show how the $q$-Krawtchouk orthogonality relations can be used to compute exponential moments and correlations of  ASEP($q,\theta$).

\section{Introduction}
In this paper we study two models of interacting particle systems  with asymmetric jump rates exhibiting a self-duality property. The first one is known in the literature as the \emph{generalized asymmetric simple exclusion process},  ASEP($q,\theta$), $\theta\in\N$ \cite{CGRS}. This is a higher spin version of the asymmetric simple exclusion process ASEP$(q)$ (corresponding to the choice  $\theta=1$)  in which particles are repelled from each other and every site can host at most $\theta\in \N$ particles.
The second process is the ASIP($q,\theta$), $\theta\in(0,\infty)$,  \emph{asymmetric simple inclusion process}, \cite{CGRS1}, where the parameter $\theta$  tunes the intensity of the attraction between particles (the smaller the $\theta$, the higher the attraction). Particles move in a finite one-dimensional lattice and  the parameter $q \in (0,1)$ tunes the asymmetry in a certain direction. In \cite{CGRS,CGRS1} a self-duality property has been shown for these models.

{Stochastic duality} is an advantageous tool used in the study of interacting particle systems that was used for the first time in \cite{spi} for the standard symmetric exclusion process (see e.g. \cite{kur,jansen2014notion,Sturm} for  surveys on the topic). Duality relations allow to connect two Markov processes via a duality function; such function is an observable   of both processes, whose expectation satisfies a specific relation.
We speak of self-duality if the two Markov processes are two copies of the same process. The usefulness of (self-)duality  is in  the fact that it allows to study the system with a large number of particles in terms of the system initialized with a finite number of particles. For example, the study of $n$ dual particles can give information on the $n$-points correlation function of the original process. Unfortunately self-duality is a  property  not always easy to reveal.

The duality function for the standard asymmetric exclusion process, ASEP$(q)$ (case $\theta=1$), and its link to quantum algebras and spin chains was first revealed in \cite{Schutz1,Schutz2}. This discovery  immediately found a vast number of    applications, allowing to find  for instance, combined with Bethe ansatz  techniques,   current fluctuations  \cite{IS} and  properties of the transition probabilities \cite{imamura1997note}.   Among other important applications of self-duality and algebraic approach for ASEP, we mention the key role played  in the study  of shocks. We mention e.g. \cite{belitsky2002diffusion} for an analysis of microscopic shock dynamics, \cite{belitsky2016self} for shocks in multispecies ASEP  and \cite{schutz2016duality} for the study of the  process conditioned to low current. The self-duality function of ASEP   is not given by a trivial product of 1-site duality functions (as in the symmetric case) but has a nested-product structure similar to the one exhibited by the  G\"artner transform  {\cite{Gartner}}.
Thanks to this structure, it  has played an important role in the proof of  convergence to the 
  KPZ equation, in the case of weak asymmetry (see e.g. \cite{bertinigiacomin,BCT,cor,corwinshentsai,imamura2011kpz}).

The partial exclusion process  in its symmetric version SEP$(\alpha)$ appeared for the first time in  
\cite{caputo2005energy} where the authors introduced it as a   particle system version of the XXX-quantum-spin-chain, with  spin higher than 1/2. 
Then the process, together with its attractive counterpart, SIP$(\alpha)$, was systematically    studied in 
\cite{giardina2007duality,kur,giardina2010correlation} were self-duality functions are found and used to prove correlation inequalities.
These processes  are  not integrable (i.e. not treatable via Bethe ansatz techniques) but   self-duality makes them  amenable  to some  analytic treatment   (see e.g. \cite{carinci2020exact}).

The asymmetric processes ASEP$(\alpha,\theta)$ and   ASIP$(\alpha,\theta)$ were finally introduced in \cite{CGRS,CGRS1} where self-duality properties are proved. These are due to the algebraic structure of the generator that is constructed passing through the
$(\alpha + 1)$-dimensional representation of a quantum Hamiltonian with $\mathcal{U}_q(\mathfrak{sl}_2)$ invariance. 
The self-duality function has again a nested-product structure, defining, in a sense,  a generalized version of the  G\"artner transform  {\cite{Gartner}}, that 
allows to  compute the  q-exponential moments of the current for suitable initial conditions.
In the last few years, several steps forward  have been done in the effort  of finding suitable multispecies versions of ASEP$(\alpha,\theta)$ showing duality properties, see e.g.  \cite{belitsky2016self,chen2018integrable,kuan2016stochastic,Kwan,kuan2019stochastic,kuan2020two}.

Most of the duality results concerning this class of processes  are {\em triangular}, i.e. are non-zero only if the dual configuration is a subset of the original process configuration. We refer to duality functions of this type also as {\em classical duality functions}. 
Orthogonal polynomial duality functions  are, on the other hand, a very  recent discovery and were found, up to date, only for symmetric processes (SEP$(\alpha)$,  SIP$(\alpha)$ and IRW) in a series of  papers \cite{CFGGR,FG,FGG,RS}.
 The duality functions for these processes are  products of univariate orthogonal polynomials, where the orthogonality is with respect to the reversible measures of the process itself.  Knowing the expectations of orthogonal polynomial  duality functions is equivalent to having all moments.  The possibility to decompose polynomial functions in $L^2(\mu)$, where $\mu$ is the reversible measure of the process,  in terms of orthogonal duality polynomials, is then a crucial property that has many repercussions in the study of macroscopic fields emerging as scaling limits of the particle system. See e.g. the work \cite{mario} for an application of orthogonal duality polynomials  for symmetric models in the study of a  generalized version of  the Boltzmann-Gibbs principle. Moreover, in two recent papers \cite{ayala2020higher,chen2020higher} orthogonal  polynomials are at the base of the definition of the so-called higher-order  fields for which the hydrodynamic limit and fluctuations are derived via duality techniques for SEP$(\alpha)$, SIP$(\alpha)$ and IRW. Finally, in a recent work \cite{FRS} orthogonal duality  results for this class of symmetric models have been extended to the non-equilibrium context, allowing to derive several properties of $n$-point correlation functions in the non-equilibrium steady state.

The families of orthogonal polynomials dualities for these processes were found for the first time in  \cite{FG} by explicit computations  relying  on the hypergeometric structure of the polynomials. The same dualities were found  in \cite{RS} via generating functions, while  an algebraic approach  is followed in \cite{wol} and \cite{CFGGR}, relying, respectively, on  the use of unitary intertwiners and unitary symmetries. In \cite{CFGGR}  yet another approach to (orthogonal) duality is described, based on scalar products of classical duality functions.

\noindent
In this  paper we use this latter approach to  extend the results obtained in \cite{CFGGR} to the case of asymmetric processes. Differently from \cite{CFGGR}, the $q$-orthogonal duality functions for asymmetric processes are not yet known in the literature. We show that well-known families of $q-$hypergeometric orthogonal polynomials, the $q-$Krawtchouk polynomials (for exclusion processes) and $q-$Meixner polynomials (for inclusion processes), occur as $1$-site duality functions for corresponding stochastic models. The $q$-orthogonal duality functions  show again a nested-product structure, as the   classical ones found in \cite{CGRS,CGRS1}, but, differently from the latter, they do not have a triangular form.
We prove that the $q-$polynomials are orthogonal with respect to the reversible measures of our models, which, in turn, have a  non-homogeneous product structure.
The nested product structure and the orthogonality relations of our duality functions are very similar to the multivariate $q-$Krawtchouk and $q-$Meixner polynomials introduced in \cite{GasRah07}, but it seems that (except for the 1-variable case) they are not the same functions.

We conjecture that the orthogonal self-duality polynomials complete the picture of nested-product duality functions for ASEP($q,\theta$) and ASIP($q,\theta$),  summing up to the classical or {\em triangular} ones, already known for these processes from \cite{CGRS,CGRS1}.
The  strategy followed in \cite{CGRS,CGRS1}  to construct the so-called classical dualities  relies on an algebraic approach  based on the study of the symmetries of the generator. This can be written, indeed, in terms of
 the Casimir operator  of the quantized enveloping algebras ${\mathcal{U}}_q(\mathfrak{su}(2))$ and  ${\mathcal{U}}_q(\mathfrak{su}(1,1))$.
The same approach was used in \cite{Kwan} for the study of duality for a multi-species version of the asymmetric exclusion process,  exploiting the link with a higher rank quantum algebra.
In the last part of the paper we will follow this algebraic approach to write (in terms of  elements of  ${\mathcal{U}}_q(\mathfrak{su}(2))$) the symmetries of the generator  yielding the $q-$polynomial dualities obtained via the scalar-product method.

\subsection{Organization of the paper}
The rest of the paper is organized as follows. In Section \ref{asipabep} we introduce the two asymmetric models of interest and their corresponding reversible measures. The dynamics takes place on a finite lattice and it is fully described by their infinitesimal generators. In particular, see Section \ref{sigma} for a unified and comprehensive notation.
In Section \ref{mainres} we recall the concept of duality for Markov processes and then exhibit the main results of this work via two theorems, Theorem \ref{teoasep} for the asymmetric exclusion and Theorem \ref{teoasip} for the asymmetric inclusion. Here the families of $q-$orthogonal polynomials, that are self-duality functions for our processes, are displayed. Besides this, we also single out those symmetries which are uniquely associated to our $q-$orthogonal polynomials. In Section \ref{sec4} we show that having a duality relation  satisfying also an orthogonal relation considerably simplifies the computation of quantities of interest, such as the $q^{-2}$ exponential moments of the current and their space time correlations. 
The rest of the paper is devoted to the proof of our main results. In Section \ref{proofofresults} we show how to obtain functions which are biorthogonal and self-dual from construction. This is done using our general Theorem \ref{prop:orthogonality}, which invokes the scalar product of classical self-duality functions. Once a biorthogonal relation is proved we show, in Section \ref{proofs} for exclusion and in Section \ref{ASIP} for inclusion, that we can easily establish an orthogonality relation by an explicit computation of the (bi)orthogonal self-duality function.
In Section \ref{Sim} we explain how we find the unique symmetries which can be used to construct our $q-$orthogonal self-duality function starting from the trivial ones. This is based on the algebraic approach used in \cite{CGRS}-\cite{CGRS1} and so Sections \ref{quantum-algebra} and \ref{proc} are inspired  by those papers in which the Markov generator is linked to the Casimir element of the algebra. In Section \ref{symm} we identify the symmetries which generate our $q-$orthogonal self-duality functions. Finally, in order to make some computations more readable, we created an Appendix, Section \ref{appendix}, where we give definitions and well-known identities regarding $q-$numbers and $q-$hypergeometric functions.

\section{The models}\label{asipabep}

In this paper we will study models of interacting particles moving on a finite lattice $\la_L=\{1,\ldots,L\}$, $L\in \N$, $L\ge 2$, with closed boundary conditions and an asymmetric interaction. 
We denote by $x=\{x_i\}_{i\in \la_L}$ (or $n=\{n_i\}_{i\in \la_L}$) a  particle  configuration where $x_i$ (resp $n_i$) is the number of particles at site $i\in \la_L$.  We call $\Omega_L=S^{L}$ the state space, where $S\subseteq \N$ is the set where the occupancy numbers $x_i$ take values. For $x\in \Omega_L$ and $i,\ell\in \la_L$ such that $x_i>0$, we denote by $x^{i,\ell}$ the configuration obtained from $x$ by removing one particle from site $i$ and putting it at site $\ell$. \\

\noindent
In this paper we will consider, in particular, two different processes: the ASEP($q,\theta$) Asymmetric Exclusion Process and  the ASIP($q,\theta$) Asymmetric Inclusion Process.
These processes share some algebraic properties even though they have  a very different behavior.
In order to define  the processes and their main properties we need to introduce some notations.

 \subsection{Notation} \label{ssec:notation}
\subsubsection*{The $q-$numbers}
Throughout the paper we fix $q\in(0,1)$ and, for $a\in \R$,  define the  $q-$numbers as follows:
\begin{equation}
\label{q-num}
 \{a\}_{q}:= \frac{1-q^{a}}{1-q}\qquad \text{and} \qquad \{a\}_{q^{-1}}:= \frac{1-q^{-a}}{1-q^{-1}}.
\end{equation}
Moreover we define
\be\label{[a]_q}
[a]_q:=\frac{q^a-q^{-a}}{q-q^{-1}} = q^{1-a} \, \{a\}_{q^2}.
\ee
Notice that, for $q\to1$,  both $[a]_q$ and $\{a\}_{q^{\pm 1}}$ converge to $a$.
Finally we define  the $q-$factorial, for $n \in \mathbb N$, given by
\be
[n]_q!:=[n]_q \cdot [n-1]_q \cdot \dots \cdot [1]_q \; \quad \text{for } n\ge 1, \qquad \text{and } \quad [0]_q!:=1.
\ee
For $\theta,m\in \N$ we define the $q-$binomial coefficient by
\begin{equation}\label{coefbin1}
\binom{\theta}{m}_q:=
\frac{[\theta]_q!}{[m]_q! [\theta-m]_q!}\; \cdot \mathbf 1_{\theta\ge m},
\end{equation}
and, for $m \in \mathbb N$ and  $\theta \in (0,\infty)$,
\begin{equation} \label{coefbin2}
\binom{m+\theta-1}{m}_q:=
\frac{[\theta+m-1]_q\cdot [\theta+m-2]_q \cdot \ldots \cdot [\theta]_q}{[m]_q!}.
\end{equation}

\subsubsection*{The $q-$Pochhammer symbol}
For $a\in \R$ and $m\in \N$ the $q-$Pochhammer symbol, or $q-$shifted factorial, $(a;q)_m$ is defined by
\be\label{Poch}
(a;q)_{m} = (1-a)(1-aq^1)\cdots (1-aq^{m-1})\;, \quad \text{for} \quad m\ge 1 \quad \text{and} \quad (a;q)_0=1 
\ee
and furthermore,
\be\label{Poch1}
(a;q)_\infty:= \prod_{k=0}^{\infty} \left( 1- aq^{k}\right) \;.
\ee
Most of the $q-$Pochhammer symbols we need in this paper depend on $q^2$ instead of $q$. To simplify notation we omit the dependence on $q$, i.e.~we write
\[
(a)_m := (a;q^2)_m \qquad \text{for} \ m \in \N \cup \{\infty\}.
\]
In light of the above, we can rewrite the $q-$factorial and the $q-$binomial coefficient in terms of the $q-$Pochhammer symbol:
\begin{equation} \label{qbinom1}
\begin{split}
[n]_q! 
= (q^{-1}-q)^{-n} q^{-\frac12n(n+1)} (q^2)_n,
\end{split}
\end{equation}
so that
\begin{equation} \label{qbinom2}
\binom{\theta}{m}_q  
= q^{-m(\theta-m)}  \frac{ (q^2)_\theta }{(q^2)_m (q^2)_{\theta-m}} \cdot \mathbf 1_{\theta\ge m}= (-1)^m q^{m(\theta+1)} \frac{ (q^{-2\theta})_m}{(q^2)_m}.
\end{equation}
Similarly,
\begin{equation} \label{qbinom3}
\binom{m+\theta-1}{m}_q = q^{-m(\theta-1)} \frac{ (q^{2\theta})_m }{(q^2)_m}.
\end{equation}
\subsection{Particle-mass functions.}\label{functionN}
For $x\in \Omega_L$, $i\in \Lambda_L$, we introduce the functions $N_i^\pm(x)$ denoting the  number of particles in the configuration $x$ at the right, respectively left, of the site $i$:
\[
N_i^+(x):=\sum_{m=i}^L x_m \qquad \text{and} \qquad N^-_i(x):=\sum_{m=1}^i x_m,
\]
with the convention that $N_{L+1}^+(x)=N^-_0(x)=0$. Moreover we denote by $N(x)$ the total number of particles in the configuration $x$:
\beq
N(x):=\sum_{m=1}^L x_m \;. \nn
\eeq
Notice that $N_L^-(x)=N^+_1(x)=N(x)$ and that these mass functions satisfy the following
 change of summation formula:
\be\label{N2}
\sum_{i=1}^{L}x_{i} N_{i-1}^-(n) = \sum_{i=1}^{L}n_{i} N^+_{i+1}(x),
\ee
moreover the following  identity holds true and will be used throughout the paper:
\be\label{N1}
\sum_{i=1}^L x_i[2N^-_{i-1}(n) + n_i] +\sum_{i=1}^L n_i[2N^-_{i-1}(x) + x_i]=  2N(x)\cdot N(n).
\ee

\subsection{The ASEP($q,\theta$)}
In the generalized Asymmetric Exclusion Process particles jump  with a repulsive interaction and each site can host at most $\theta$ particles, where $\theta$ is now a parameter taking values in $\N$. Hence, in this case  $S=\{0,1,\ldots, \theta\}$ and $\Omega_L=S^{L}$. In the usual asymmetric simple exclusion process each site can either be empty or host one particle, while here each site can accommodate up to $\theta$ particles. Hence, by setting $\theta$ equal to $1$ we recover the hard-core exclusion. The infinitesimal generator is presented in the following definition.

\bd[ASEP($q,\theta$)]
 The \textup{ASEP}$(q,\theta)$ with closed boundary conditions is defined
as the Markov process on $\Omega_L$ with generator  $\caL^{\text{\tiny{ASEP}}}=\caL^{\text{\tiny{ASEP}}}_{(L)}$
defined on functions $f:\Omega_L\to\R$ by
\[
[{\cal L}^{\textup{\tiny{ASEP}}}f](x):=\sum_{i=1}^{{L-1}} [{\cal L}^{\textup{\tiny{ASEP}}}_{i,i+1}f](x) \qquad \text{with}
\]
\beq
[{\cal L}^{\textup{\tiny{ASEP}}}_{i,i+1}f](x) &
 := & q^{-2\theta +1} \{x_i\}_{q^2} \{\theta-x_{i+1}\}_{q^{2}} [f(x^{i,i+1}) - f(x)] \nn\\
&+ &  q^{2\theta -1} \{x_{i+1}\}_{q^{-2}} \{\theta-x_{i}\}_{q^{-2}}[f(x^{i+1,i}) - f(x)]. \nn
\eeq
\ed

\subsubsection*{Reversible  signed measures}
From Theorem 3.1 of \cite{CGRS} we know that the ASEP($q,\theta$) on $\la_L$ with closed boundary conditions
admits a family, labeled by $\alpha \in \R\setminus \{0\}$, of reversible product, non-homogeneous signed measures $\mu_\alpha^{\text{\tiny{ASEP}}}$ given by
\be
\label{stat-measasep}
{\mu_\alpha^{\textup{\tiny{ASEP}}}}(x) = \prod_{i=1}^L \alpha^{x_i} \,{\binom{\theta}{x_i}_q} \cdot  q^{- 2\theta  ix_i}  
\ee
for $i \in \la_L$. For positive values of $\alpha$, \eqref{stat-measasep} can be interpreted, after renormalization,  as a probability measure. Here the normalizing constant is $\prod_{i=1}^L Z_{i,\alpha}^{\text{\tiny{ASEP}}}$, with
\[
Z_{i,\alpha}^{\text{\tiny{ASEP}}} = \sum_{m=0}^{\theta}  {\binom{\theta}{m}_q} \cdot  \alpha^m q^{- 2\theta  i m}=(-\alpha q^{1-\theta(2i+1)})_{\theta},\qquad \text{for }\quad \alpha \in (0, \infty)
\]
where the identity follows from the $q-$binomial Theorem \eqref{newton}. In order to make sense of the constant $\alpha$ labelling the measure, one may e.g. compute the $q$-exponential moment w.r. to the normalized measure $\bar\mu_\alpha:= \mu_\alpha/Z_{\alpha}$:
\be
\mathbb E_{\bar \mu_\alpha}\left[q^{2x_i}\right]=\frac{(-\alpha q^2 q^{1-\theta(2i+1)})_{\theta}}{(-\alpha q^{1-\theta(2i+1)})_{\theta}}=\frac{1+\alpha q^{1-2\theta i+\theta}}{1+\alpha q^{1-2\theta i-\theta}}
\ee  
where we used the identity (1.8.11) in \cite{Koekoek}.

\subsection{The ASIP($q,\theta$) }
The Asymmetric Inclusion Process is a model in which  particles jump  with an attractive interaction. The parameter $\theta>0$ tunes the intensity of the interaction, the higher the attractiveness the smaller the $\theta$. Each site of the lattice $\la_L$ can host an arbitrary number of particles, thus, in this case we have $S=\N$ and then $\Omega_L=\N^{L}$.
We introduce the process by giving its generator.
\bd[ASIP($q,\theta$)]
\label{defASIP}
\noindent
The \textup{ASIP}$(q,\theta)$ with closed boundary conditions is defined
as the Markov process on $\Omega_L$ with generator $\caL^{\text{\tiny{ASIP}}}=\caL^{\text{\tiny{ASIP}}}_{(L)}$
defined on functions $f:\Omega_L\to\R$ by
\be\label{gen}
[{\cal L}^{\textup{\tiny{ASIP}}}f](x):=\sum_{i=1}^{{L-1}} [{\cal L}^{\textup{\tiny{ASIP}}}_{i,i+1}f](x) \qquad \text{with}\nn
\ee
\beq
[{\cal L}^{\textup{\tiny{ASIP}}}_{i,i+1}f](x) &
 := & q^{2\theta -1} \{x_i\}_{q^2} \{\theta +x_{i+1}\}_{q^{-2}} [f(x^{i,i+1}) - f(x)] \nn\\
&+ &  q^{-2\theta +1} \{x_{i+1}\}_{q^{-2}} \{\theta +x_{i}\}_{q^{2}}[f(x^{i+1,i}) - f(x)].
\eeq
\ed
\noindent
Since in finite volume we always start with finitely many particles, and the total particle number is conserved, the process is automatically well defined as a finite state space continuous time Markov chain.

\subsubsection*{Reversible  signed measures}
It is proved in Theorem 2.1 of \cite{CGRS1} that the ASIP($q,\theta$) on $\Lambda_L$ with closed boundary conditions
admits a family labeled by $\alpha\in \R\setminus \{0\}$ of reversible product non-homogeneous signed measures $\mu_\alpha^{\text{\tiny{ASIP}}}$  given by
\be
\label{stat-measasip}
{\mu_\alpha^{\text{\tiny{ASIP}}}}(x) =\prod_{i=1}^L\alpha^{x_i} \,{\binom{x_i+\theta -1}{x_i}_q} \cdot  q^{2\theta ix_i},  
\ee
for $x \in \N^L$. Restricting to positive values of the parameter $\alpha$, this can be turned to  a probability measure after renormalization that is possible only under the further restriction $\alpha <q^{-(\theta +1)}$.   In order to normalize it we should divide by
 the   constant $\prod_{i=1}^L Z_{i,\alpha}^{\text{\tiny{ASIP}}}$, with
\be \label{Z}
Z_{i,\alpha}^{\text{\tiny{ASIP}}}= \sum_{m=0}^{+\infty}  {\binom{m+\theta -1}{m}_q} \cdot  \alpha^m q^{2\theta i m} = \frac{(\alpha q^{2\theta i+\theta +1})_{\infty}}{(\alpha q^{2\theta i-\theta +1})_{\infty}}, \qquad \text{for }\quad \alpha \in (0, q^{-(\theta +1)})
\ee
where the latter identity follows from the $q-$binomial Theorem, \cite[(II.3)]{gasper2004basic}.
However, to keep notation light we work with the non-normalized measure. Also in this case one can easily compute the $q$-exponential moment w.r. to $\bar\mu_\alpha:= \mu_\alpha/Z_{\alpha}$:
\be
\mathbb E_{\bar \mu_\alpha}\left[q^{2x_i}\right]= \frac{(\alpha q^2 q^{2\theta i+\theta +1})_{\infty}}{(\alpha q^{2\theta i+\theta +1})_{\infty}}\cdot  \frac{(\alpha q^{2\theta i-\theta +1})_{\infty}}{(\alpha  q^2 q^{2\theta i-\theta +1})_{\infty}}=\frac{1-\alpha q^{1+2\theta i-\theta}}{1-\alpha q^{1+2\theta i+\theta}}
\ee  
where, for the second identity we used eq.(1.8.8) in \cite{Koekoek}.

\subsection{General case}\label{sigma}
In order to simplify the notation it is convenient to introduce  a parameter $\sigma$ taking values in $\{-1,+1\}$, distinguishing between the two cases: $\sigma=+1$ corresponding to the inclusion process and $\sigma=-1$ corresponding to the exclusion process. In what follows, if needed, we will omit the superscripts ASIP or ASEP and simply denote by $\caL$ the generator of one of the processes, meaning
\be\label{LL}
\caL_\sigma=
\left\{
\begin{array}{ll}
\caL^{\text{\tiny{ASEP($q,\theta$)}}} & \text{for } \sigma=-1\\
\caL^{\text{\tiny{ASIP($q,\theta$)}}} & \text{for } \sigma=+1.
\end{array}
\right.
\ee
where the parameter $\theta$ takes values in  $\N$ for $\sigma=-1$ and in $(0,\infty)$ for $\sigma=1$.
Particles occupation numbers take values in
\be
S_{\sigma,\theta}:=
\left\{
\begin{array}{ll}
\{0,1, \ldots, \theta\}& \text{for }\sigma=-1,\\
\N & \text{for }\sigma=+1,
\end{array}
\right.
\ee
and the state space of the process is $\Omega_L:=S_{\sigma,\theta}^L$.
 We can then write the generator (for the bond $i,i+1 \in \la_L$) in the general form:
\beq
[{\cal L}_{i,i+1}f](x)
& := &  q^{\sigma(2\theta-1)} \{x_i\}_{q^2} \{\theta+\sigma x_{i+1}\}_{q^{-2\sigma}} [f(x^{i,i+1}) - f(x)] \nn\\
&+ &   q^{-\sigma(2\theta-1)} \{x_{i+1}\}_{q^{-2}} \{\theta +\sigma x_{i}\}_{q^{2\sigma}}[f(x^{i+1,i}) - f(x)]. \nn
\eeq
Then, defining the function
\beq\label{Psi}
\Psi_{q,\sigma}(\theta,m):= (\sigma q)^m \; q^{-\sigma\theta m} \; \frac{(q^{2\sigma\theta})_m}{(q^2)_m}=
\left\{
\begin{array}{ll}
\binom{\theta}{m}_q & \text{for } \sigma=-1\\
&\\
\binom{m+\theta-1}{m}_q & \text{for } \sigma=+1
\end{array}
\right.
\eeq
(see \eqref{qbinom2}-\eqref{qbinom3}) the reversible  signed measure \eqref{stat-measasep}-\eqref{stat-measasip} can be rewritten in a unique expression as follows
\be
\label{SM}
\mu_{\alpha,\sigma}(x) = \prod_{i=1}^L\Psi_{\sigma}(\theta,x_i)\cdot  \alpha^{x_i} \, q^{2\sigma\theta  ix_i}, \qquad \text{for} \quad \alpha\in \R\setminus \{0\}
\ee
for $x\in\Omega_L=S_{\sigma,\theta}^L$.
\vskip.2cm
\noindent
  We define a modified version $\omega_{\alpha,\sigma}$ of \eqref{SM} that will appear in the statement of the main results in Section \ref{mainres}.
  This new signed measure differs from \eqref{SM} only through  multiplication  by a function of the total number of particles $N(x)$:
\beq\label{w}
\omega_{\alpha,\sigma}(x)=\frac{\mu_{\alpha,\sigma}(x)}{{\cal Z}_{\alpha,\sigma}}\cdot q^{N(x)(N(x)-1)} \cdot  (-\sigma \alpha q^{1+2N(x) +\sigma\theta(2 L+1)})_\infty
\eeq
where ${\cal Z}_{\alpha,\sigma}$ is a constant.
We remark that, as the processes conserve the total number of particles,   detailed balance condition is preserved under this operation, then  $\omega_{\alpha,\sigma}$ is  again a reversible signed measure for the processes. In order to interpret it as a probability measure we have to restrict to the case $\alpha>0$.  This condition is sufficient for the case $\sigma=+1$, while, for $\sigma=-1$ we have to impose the further condition $\alpha<q^{-1+(2L+1)\theta}$ in order to assure the positivity of the infinite $q-$shifted factorials. Under these conditions and choosing
\be
{\cal Z}_{\alpha,\sigma}:=\sum_{x}{\mu_{\alpha,\sigma}(x)}\cdot q^{N(x)(N(x)-1)} \cdot  (-\sigma \alpha q^{1+2N(x) +\sigma\theta(2 L+1)})_\infty.
\ee
$\omega_{\alpha,\sigma}$ is a reversible probability measure for the corresponding process:
\be\label{W}
\begin{array}{lll}
\omega^{\text{\tiny{ASEP($q,\theta$)}}}:=\omega_{\alpha,-1},& \text{for }  \alpha\in(0,q^{-1+(2L+1)\theta})\\
&&\\
\omega^{\text{\tiny{ASIP($q,\theta$)}}}:= \omega_{\alpha,+1},& \text{for }  \alpha \in (0, q^{-1-\theta}).\\
\end{array}
\ee
Finally we define the function:
\be\label{tw}
g_{\alpha,\sigma}(x):= \frac{q^{2N(x)(N(x)-1)}} {{\cal Z}^2_{\alpha,\sigma}}\cdot  (-\sigma \alpha q^{1+2N(x) +\sigma\theta(2 L+1)})_\infty\cdot(-\sigma\alpha q^{1-2N(x)+\sigma\theta})_\infty 
\ee
that will also appear in the statement of the main results.


\section{Main Results} \label{mainres}

The main result of this paper is the proof of self-duality properties for the processes
introduced in the previous section via $q-$hypergeometric orthogonal polynomials. For each process we show the existence of a self-duality function, $D$ and another one, $\widetilde D$, that is the same modulo multiplication by a function of the total number of particles and the size of the lattice. Such duality functions can be written in terms of the  $q-$Krawtchouk polynomials (respectively $q-$Meixner polynomials) for the ASEP($q,\theta$) (respectively for the  ASIP($q,\theta$)). $D$ and  $\widetilde D$  satisfy a biorthogonality relation if one considers the scalar product with respect to the (one site) reversible measures. However, the biorthogonal relation can easily be stated as an orthogonal relation by performing the change of measure of equation \eqref{w} and the consequently change of norm in equation  \eqref{w}.
We start by recalling below the definition of duality.
\bd
\label{standard}
Let  $\{X_t\}_{t\ge0}$,  $\{\widehat{X}_t\}_{t\ge0}$ be two  Markov processes with state spaces  $\Omega$ and $\widehat{\Omega}$ and $D: \Omega\times \widehat{\Omega}\to\R$ a measurable function. The processes  $\{X_t\}_{t\ge0}$,  $\{\widehat{X}_t\}_{t\ge0}$ are said to be \textbf{dual} with respect to $D$  if
\be\label{standarddualityrelation1}
\E_x \big[D(X_t, \widehat{x})\big]=\widehat{\mathbb{E}}_{\widehat{x}} \big[D(x, \widehat{X}_t)\big]\;
\ee
for all $x\in\Omega, \  \widehat{x}\in \hat{\Omega}$ and $t>0$. Here $\E_x $ denotes the expectation with respect to the
 law of the process $\{X_t\}_{t\ge0}$  started at $x$, while $\widehat{\mathbb{E}}_{\widehat{x}} $
denotes  expectation with respect to the law of the process $\{\widehat{X}_t\}_{t\ge0}$
initialized at $\widehat{x}$. If $\{\widehat{X}_t\}_{t\ge0}$  is a copy of $\{X_t\}_{t\ge0}$, we say that the process $\{X_t\}_{t\ge0}$ is \textbf{self-dual}.
\ed

\subsection*{Orthogonal polynomial dualities for ASEP($q,\theta$)}
In this section we display the orthogonal duality function for  ASEP($q,\theta$), namely the $q-$Krawtchouk polynomials, for which we will use the following notation
\begin{equation} \label{def:q-Krawtchoukbis}
K_n(q^{-x};p,c;q):=  \rphis{2}{1}{q^{-x},q^{-n}}{q^{-c}}{q,p{q^{n+1}}},
\end{equation}
where $_{2}\phi_{1}$ is the  $q-$hypergeometric function, $c \in \N$ and $n,x \in \lbrace 0,\ldots, c \rbrace$, see Section \ref{Kraw} of the Appendix for the orthogonality relations.
The following theorem states that nested products of $q-$Krawtchouk polynomials form a family of self-duality functions for ASEP($q,\theta$).
\bt \label{teoasep}
The \textup{ASEP}($q,\theta$) on $\Lambda_L$ is  self-dual  with self-duality functions:
\beq\label{Dasep}
 &&D_\alpha^{\textup{\tiny{ASEP($q,\theta$)}}}(x,n) :=  \prod_{i =1}^L  K_{n_i}(q^{-2x_i};p_{i,\alpha}(x,n),\theta ;q^2), \qquad  \alpha\in(0,q^{-1+(2L+1)\theta}),\nn\\
&& p_{i,\alpha}(x,n):=\frac 1 {\alpha}\; q^{-2(N_{i-1}^-(x)-N^+_{i+1}(n))+\theta(2i-1)-1}
\eeq
 satisfying the following orthogonality relation
\be\label{BiOrtasep}
\langle D_\alpha^{\textup{\tiny{ASEP($q,\theta$)}}}(\cdot , x), D_{\alpha}^{\textup{\tiny{ASEP($q,\theta$)}}}(\cdot , n)\rangle_{\omega_\alpha^{\textup{\tiny{ASEP($q,\theta$)}}}}=\dfrac{\delta_{x,n}}{ \omega_\alpha^{\textup{\tiny{ASEP($q,\theta$)}}}(x)}\cdot g_{\alpha,-1}(x)
\ee
with $\omega_\alpha^{\textup{\tiny{ASEP($q,\theta$)}}}$ and $ g_{\alpha,-1}$
 defined  in \eqref{W}-\eqref{tw}.
\et
\begin{remark} \*
For $L=1$ this gives the orthogonality relations for $q-$Krawtchouk polynomials as stated in Section \ref{Kraw}, so we have obtained a family of multivariate orthogonal polynomials generalizing the $q-$Krawtchouk polynomials.
Note that the restriction  $\alpha\in(0,q^{-1+(2L+1)\theta})$ has been imposed in order to have a scalar product  \eqref{BiOrtasep} w.r. to a (positive) reversible measure, that can be eventually turned in a probability measure, after renormalization.
Note also that this is the condition required  in order to have the conditions \eqref{condK} satisfied, indeed, for $\alpha\in(0,q^{-1+(2L+1)\theta})$,
\be
\theta \in \N \qquad  \text{and} \quad q^{2\theta} \cdot p_{i,\alpha}(x,n)>1 \qquad  \text{for all} \quad  x,n\in \Omega_L, \: i \in \Lambda_L \;.
\ee
If we neglect this condition Theorem \ref{teoasep} holds still true  with the only difference that we can not guarantee the positivity of $\omega_\alpha$.
\end{remark}

\subsection*{Orthogonal polynomial dualities for ASIP($q,\theta$)}

In the same spirit of the previous section we now introduce the orthogonal duality relation for ASIP($q,\theta$). In this case we have that the self-duality functions are a nested product of $q-$Meixner polynomials
\beq \label{q-Meixnerbis}
M_n(q^{-x};b,c;q):=  \rphis{2}{1}{q^{-x},q^{-n}}{bq}{q,-\frac{q^{n+1}}{c}}, \qquad \text{for} \quad x,n \in \N,
\eeq
see Section \ref{Meix} in the Appendix for more details and orthogonality relations. The following theorem is the analogue of the previous one; it says that a family of nested $q-$Meixner polynomials are self-duality functions for ASIP($q,\theta$).
\bt \label{teoasip}
The \textup{ASIP}($q,\theta$) on $\Lambda_L$ is  self-dual  with  self-duality functions
\beq\label{D}
&& D_\alpha^{\textup{\tiny{ASIP($q,\theta$)}}}(x,n) :=  \prod_{i =1}^L M_{n_i}(q^{-2x_i};q^{2(\theta -1)}, c_{i,\alpha}(x,n);q^2), \qquad \alpha>0\nn\\
&&
 c_{i,\alpha}(x,n):=\alpha q^{2(N_{i-1}^-(x)-N^+_{i+1}(n))+\theta (2i-1)+1},
\eeq
for all  $\alpha\in (0,q^{-1-\theta})$, satisfying the following orthogonality relations
\be\label{BiOrtasip}
\langle D_\alpha^{\textup{\tiny{ASIP($q,\theta$)}}}(\cdot , x), D_{\alpha}^{\textup{\tiny{ASIP($q,\theta$)}}}(\cdot , n)\rangle_{\omega_{\alpha}^{\textup{\tiny{ASIP($q,\theta$)}}}}=\dfrac{\delta_{x,n}}{\omega_{\alpha}^{\textup{\tiny{ASIP($q,\theta$)}}}(x)}\cdot  g_{\alpha,1}(x)
\ee
with $\omega_\alpha^{\textup{\tiny{ASIP($q,\theta$)}}}$
and $ g_{\alpha,1}$
 defined  in \eqref{W}-\eqref{tw}.
\et
\noindent
\begin{remark}
For $L=1$ this gives the orthogonality relations for $q-$Meixner polynomials as stated in Section \ref{Meix}. The orthogonal polynomials have a similar structure as the multivariate $q-$Meixner polynomials introduced in \cite{GasRah07}, but they do not seem to be the same functions.  Notice that the conditions \eqref{condM}  are satisfied, indeed
\beq
q^{2\theta}\in [0,1) \qquad \text{and} \qquad c_{i,\alpha}(x,n)>0 \qquad \text{for all}\quad x,n \in \Omega_L, \: i\in \Lambda_L
\eeq
if $\alpha>0$. As in the case of ASEP, the condition $\alpha>0$ is only needed in order to assure the positivity of the measure $\omega_\alpha$.
\end{remark}

\subsection*{Orthogonal self-dualities and symmetries}
Whenever the process is reversible
it has now been established that there is a one-to-one correspondence between self-duality (in the context of Markov process with countable state space) and symmetries of the Markov generator. The idea is the following: the reversible measure of our processes provides a \textit{trivial} self-duality function (which is the inverse of the reversible measure itself). Then the action of a symmetry of the model on this trivial self-duality gives rise to a non-trivial self-duality function, see \cite{CFGGR} (Section 2.3) or \cite{kur}.
For this reason it is natural to ask which are the symmetries associated to our orthogonal self-dualities. In the context of orthogonal polynomials, we know that the symmetries must preserve the norm of the trivial self-duality function, i.e. the symmetry is unitary. Recall that a unitary operator on the space $L^{2}(\mu)$ is such that its adjoint corresponds to its inverse.
In order to recover the unitary symmetries associated to the orthogonal dualities  we first {\em normalize} the self-duality functions  \eqref{Dasep} and \eqref{D}. At this aim we define
\be \label{normalizedselfduality}
\D_{\alpha,\sigma}(x,n):=D_{\alpha,\sigma}(x,n)\cdot  q^{\binom{N(x)}{2}-\binom{N(n)}{2}}
\cdot \sqrt{\frac{ (-\sigma \alpha q^{1+2N(x) +\sigma\theta(2 L+1)})_\infty
}{(-\sigma\alpha q^{1-2N(n)+\sigma\theta})_\infty }}
\ee
with
\be\label{DDD}
D_{\alpha,\sigma}=\left\{
\begin{array}{ll}
D_\alpha^{\text{\tiny{ASEP($q,\theta$)}}}& \text{for }\sigma=-1\\
D_\alpha^{\text{\tiny{ASIP($q,\theta$)}}}& \text{for }\sigma=+1\\
\end{array}
\right.
\ee
Notice that the functions $\D_{\alpha,\sigma}$, with $\sigma=\pm 1$ are equal to the old dualities modulo multiplication by a factor that only depends on the total number of particles in both configurations. As a consequence the functions $\D_{\alpha,\sigma}$ are themselves a family of self-duality functions as the dynamics conserves the mass (see e.g. Lemma 3 of \cite{CFGGR}).
After this renormalization the orthogonality relations read
\be\label{Ort}
\langle \D_{\alpha,\sigma}(\cdot , x), \D_{\alpha,\sigma}(\cdot , n)\rangle_{\mu_{\alpha,\sigma}}=\dfrac{\delta_{x,n}}{ \mu_{\alpha,\sigma}(x)} \;.
\ee
We can reinterpret now the orthogonal self-duality function $\D_{\alpha,\sigma}$ as the result of the action of a unitary symmetry $\s_{\alpha,\sigma}$ of the generator on the trivial duality function constructed as the inverse of the reversible measure i.e. $\dfrac{\delta_{x,n}}{ \mu_{\alpha,\sigma}(x)}$. More precisely, as a consequence of the above, defining
\beq\label{S}
\s_{\alpha,\sigma}(x,n):=\D_{\alpha,\sigma}(x,n)\cdot \mu_{\alpha,\sigma}(n)
\eeq
we have the following result {for $\sigma=-1$}.
\bp
For $\alpha >0$, we have that $\s_{\alpha,-1}$ is a symmetry of the generator $\caL_{-1}$ defined in \eqref{LL}, i.e. $[\s_{\alpha,-1}, \caL_{-1}]=0$. Moreover it is a unitary operator in $L^2(\mu_\alpha)$ i.e. $\s_{\alpha,-1}^*\s_{\alpha,-1}=\s_{\alpha,-1}\s_{\alpha,-1}^*=I$.
\ep
\br
For $\sigma=+1$, both $\caL_{+1}$ and $\s_{\alpha,+1}$ are unbounded operators on $L^2(\mu_\alpha)$. If we choose the set of finitely supported functions in $L^2(\mu_\alpha)$ as a dense domain for both operators, they commute on this domain. We do not have unitarity of $\s_{\alpha,+1}$. The relation $\s_{\alpha,+1}^*\s_{\alpha,+1}=I$ holds because this is equivalent to the orthogonality relations for $D_\alpha$. But the relation $\s_{\alpha,+1}\s_{\alpha,+1}^*=I$ does not automatically follow from this as in the finite dimensional setting. In fact, the latter relation is not valid, which is a consequence of the fact that the $q$-Meixner polynomials do not form a complete orthogonal set in their weighted $L^2$-space.
In the last part of Appendix \ref{Meix} we address this issue.
\er
}
\noindent
In Section \ref{Sim} we will give an expression of the symmetry $\s_{\alpha, -1}$ in terms of the generators of the quantized enveloping algebra {${\mathcal{U}}_q(\mathfrak{sl}_2)$, in the spirit of \cite{CGRS}-\cite{CGRS1}. In order to do this we will pass through the construction of the generator of the processes from a quantum Hamiltonian, that is in turn  built from the coproduct of the Casimir operator of   {${\mathcal{U}}_q(\mathfrak{sl}_2)$.
\vskip.4cm
\noindent
\br[Symmetric case]
Performing the limiting relation as $q\rightarrow 1$ then the families of hypergeometric $q-$orthogonal polynomials converge to the classical hypergeometric orthogonal polynomials found in \cite{CFGGR}, which are families of self-duality functions for the corresponding symmetric interacting particle systems. In this limit the duality functions lose their nested-product structure and become ordinary product functions.
\er

\br[Space of self-duality functions]
A question that naturally arises regards the space of self-duality for our asymmetric models. In the symmetric setting, it has been established in \cite{RS} that, up to constant factors, the only possible product self-duality functions are the trivial, the classical and the orthogonal ones. We conjecture that in the asymmetric case one can make a similar characterization under the assumption of a nested product form. However  a rigorous proof could be an interesting subject for a future work. 
\er

\section{Duality moments and correlations}\label{sec4}
In this section we show how the duality relation can be used to compute suitable  moments and correlations of the process. 

\vskip.2cm
\noindent
In this section we will use the generic notation   $\{x(t), \, t\ge 0\}$ and $\{n(t), \, t\ge 0\}$  to denote two copies of the process with generator $\cal L_\sigma $ defined in \eqref{LL} with state space $\Omega_L=S_{\sigma,\theta}^L$. This process corresponds to  ASEP$(q,\theta)$ for $\sigma=-1$ and to ASIP$(q,\theta)$ for $\sigma=1$.
  We denote by
$\mathbb P_x$, resp. $\mathbb E_x$, the probability measure, resp. expectation, of one copy of the process conditioned to the initial value $x(0)=x$.
The duality relation \eqref{standarddualityrelation1} reads as
\be\label{DuRel}
\mathbb E_x \left[D_{\alpha,\sigma}(x(t),n)\right]=\mathbb E_n \left[D_{\alpha,\sigma}(x,n(t))\right],
\ee
that holds true for the duality function $D_{\alpha,\sigma}$ defined in \eqref{DDD}-\eqref{Dasep}-\eqref{D}.
Thinking now the original process $\{x(t), \, t\ge 0\}$  as a process with a high number of particles and the dual one $\{n(t), \, t\ge 0\}$ as a process with a few particles, and calling 
$n$-th duality moment at time $t$ the expectation $\mathbb E_x \left[D_{\alpha,\sigma}(x(t),n)\right]$, relation \eqref{DuRel} tells us that
it is possible to compute the duality moments of the original process in terms of the dynamics of   $\|n\|$ dual particles.
The added value of the orthogonality relation  lies in the possibility of computing the stationary two-times correlations. 

\vskip.4cm
\noindent
{\bf Two-times correlations.}\label{TWO}
For the case $\sigma=-1$, the duality functions $\{D_{\alpha,-1}(\cdot,n), \; n \in \Omega_L\}$ form a basis of $L^2(\omega_{\alpha,-1})$, where $\omega_{\alpha,-1}$ is the reversible probability measure of ASEP$(q,\theta)$ defined in \eqref{W}. It follows that any function $f\in L^2(\omega_{\alpha,-1})$  can be expanded in terms of the duality polynomials $D_{\alpha,-1}(\cdot,n)$,   $n \in \Omega_L$. The same does not hold true for $\sigma=1$ since $\{D_{\alpha,1}(\cdot,n), \; n \in \Omega_L\}$
does not form a basis of $L^2(\omega_{\alpha,-1})$. As a consequence, for the latter case only for functions in the span of $\{D_{\alpha,1}(\cdot,n), \; n \in \Omega_L\}$ one can get a similar expansion. 
\vskip.4cm
\noindent
In general we have that, for any fixed
\be
\alpha \in (0,\alpha_\sigma^{\text{max}}), \qquad \text{with} \qquad \alpha_\sigma^{\text{max}}=
\left\{
\begin{array}{ll}
q^{-1+(2L+1)\theta}& \text{for }\sigma=-1\\
q^{-1-\theta}& \text{for }\sigma=+1\\
\end{array}
\right.
\ee
and for any
 $f\in \text{Span}\{D_{\alpha,\sigma}(\cdot,n), \; n \in \Omega_L\}$, we have
\be
f=\sum_{n\in \Omega_L} C_f(n)\cdot D_{\alpha,\sigma}(\cdot,n)\label{expa}
\ee
where, from \eqref{BiOrtasep}-\eqref{BiOrtasip}, $C_f$ is given by:
\be\label{Cf}
C_f(n)=  \frac{\omega_{\alpha,\sigma}(n)}{g_{\alpha,\sigma}(n)}\cdot \langle f, D_{\alpha,\sigma}(\cdot,n) \rangle_{\omega_{\alpha,\sigma}}.
\ee
with $g_{\alpha,\sigma}$ the function defined in \eqref{tw}.  This orthogonal expansion substantially simplifies  the computation of the two-times correlations as shown in the following Theorem. 
\vskip.3cm
\bt\label{Teo:Corr}
Fix  $\alpha\in(0,\alpha_\sigma^{\text{max}})$ and let $f,h\in \text{\it Span}\{D_{\alpha,\sigma}(\cdot,n), \; n \in \Omega_L\}$, then, for all $0\le s\le t$, 
\be\label{Corr}
\mathbb E_{\omega_{\alpha,\sigma}}\left[ h(x(s))\cdot f(x(t))\right]= \sum_{n,m\in \Omega_L} \frac{\omega_{\alpha,\sigma}(n)}{g_{\alpha,\sigma}(n)} \cdot  C_h(m)C_f(n) \cdot  \mathbb P_n(n(t-s)=m).
\ee
\et
\bpr
In this proof we will omit the subscript $\sigma$. We have
\beq\label{here}
&&\mathbb E_{\omega_\alpha}\left[ h(x(s))\cdot f(x(t))\right]=\nn\\
&&= \sum_{n,m\in \Omega_L}\frac{\omega_\alpha(m)}{g_\alpha(m)}\cdot\frac{\omega_\alpha(n)}{g_\alpha(n)}\cdot C_h(m)C_f(n) \cdot \mathbb E_{\omega_\alpha}\left[D_\alpha(x(s),m) D_\alpha(x(t),n) \right]
\eeq
where
\beq
&& \mathbb E_{\omega_\alpha}\left[D_\alpha(x(s),m) D_\alpha(x(t),n) \right]=\sum_{x\in \Omega_L} \omega_\alpha(x) \cdot \mathbb E_{x}\left[D_\alpha(x(s),m) D_\alpha(x(t),n) \right]\nn\\
&&=\sum_{x\in \Omega_L} \omega_\alpha(x) \cdot \mathbb E_{x}\left[D_\alpha(x(s),m) \cdot \mathbb E_{x(s)}[D_\alpha(x(t-s),n)] \right]\nn\\
&&=\sum_{x\in \Omega_L} \omega_\alpha(x) \cdot D_\alpha(x,m) \cdot \mathbb E_{x}[D_\alpha(x(t-s),n)] \nn\\
&&=\sum_{x\in \Omega_L} \omega_\alpha(x) \cdot D_\alpha(x,m)  \cdot \mathbb E_{n}[D_\alpha(x,n(t-s))] \nn \\
&&=\mathbb E_{n}\left[\sum_{x\in \Omega_L} \omega_\alpha(x) \cdot D_\alpha(x,m) \cdot D_\alpha(x,n(t-s))\right] \nn \\
&&=\frac{g_\alpha(m)}{\omega_\alpha(m)}\cdot  \mathbb P_n(n(t-s)=m) \nn 
\eeq
where the second identity follows from the Markov property, the third one from the stationarity of $\omega_\alpha$, the forth one from duality, and the last one from \eqref{BiOrtasep}-\eqref{BiOrtasip}.
Then, using \eqref{here}, we get \eqref{Corr}.
\epr
\vskip.3cm
\noindent
A similar result holds true for symmetric system, see for instance  Section 3.3 of \cite{mario}  where an expansion of the type of \eqref{expa}  has been used to derive  an higher-order version of the Boltzmann-Gibbs principle, for a system of independent random walkers. An analogous identity holds true also for SEP$(\theta)$ and SIP$(\theta)$.
In, general, for this whole class of symmetric models admitting orthogonal polynomial dualities,  the symmetric version of Theorem  \ref{Teo:Corr} allows to compute the two-times correlations of the duality observables (see e.g. equation (16) in \cite{ayala2020higher}).   This identity has been a crucial ingredient in the definition and study of the so-called ``higher-order'' density fields \cite{ayala2020higher,chen2020higher} for which  a full characterization of the hydrodynamic and fluctuations scaling limits has been achieved thanks to orthogonal dualities. 
\vskip.5cm
\noindent
{\bf $q^{-2}$-exponential moments.}
In order to apply  Theorem \ref{Teo:Corr} we need to detect the functions $f\in \text{Span}\{D_{\alpha,\sigma}(\cdot,n), \; n \in \Omega_L\}$ for which the coefficients $C_f(\cdot)$, or, equivalently, the projections $\langle f, D_{\alpha,\sigma}(\cdot,n) \rangle_{\omega_{\alpha,\sigma}}$ can be easily computed.  The most natural example of such functions is $f:=q^{-2N^{-}_{i}(\cdot)}$. Indeed one can easily check by direct computation that, choosing $n=\delta_i$ for some $i\in \Lambda_L$, one has:
\begin{equation} \label{m=1}
D_\alpha(x,\delta_i)  =1- \dfrac{q^{-2\sigma \theta i +1}}{\left( q^{\theta} - q^{-\theta}\right) \alpha } \left[ q^{-2N^{-}_{i-1}(x)} - q^{-2N^{-}_{i}(x)} \right] 
\end{equation}
and, as a consequence,
\be\label{first}
q^{-2N^{-}_{i}(x)}
=1+\sigma  \alpha q^{\sigma\theta-1}\left(1 - q^{2\sigma\theta i}\right)  +  \alpha q^{-1}\left( q^{\theta} - q^{-\theta}\right) \sum_{\ell=1}^i {  {q^{2\sigma \theta \ell }}}\cdot  D_\alpha(x, \delta_\ell)
\ee
from which it follows that 
\beq
&& C_f({\bf 0})=1+\sigma  \alpha q^{\sigma\theta-1}\left(1 - q^{2\sigma\theta i}\right)  , \qquad 
C_f(\delta_\ell)= \alpha q^{-1}\left( q^{\theta} - q^{-\theta}\right) q^{2\sigma \theta \ell}\cdot \mathbf 1_{\lbrace \ell\le i \rbrace},\nn\\
&& \qquad C_f(n)=0 \quad \text{for all }n: \: \|n\|>1, \nn
\eeq
Then, using Theorem  \ref{Teo:Corr}, we obtain the following formula for the space-time correlations of the $q^{-2}$-exponential moments of $N_i^-(x)$:
\beq\label{covv}
&&\mathbb E_{\omega_{\alpha,\sigma}}\left[q^{-2N^{-}_{i}(x(s))}\cdot q^{-2N^{-}_{j}(x(t))}\right]=\nn\\
&&= \frac{\omega_{\alpha,\sigma}({\bf 0})}{g_{\alpha,\sigma}({\bf 0})} \,
 \(1+\sigma  \alpha q^{\sigma\theta-1}\left(1 - q^{2\sigma\theta i}\right)  \) \(1+\sigma  \alpha q^{\sigma\theta-1}\left(1 - q^{2\sigma\theta j}\right) \)\nn \\
&&+ \alpha^2 q^{-2}\left( q^{\theta} - q^{-\theta}\right)^2\sum_{\ell=1}^j \sum_{\kappa=1}^i \frac{\omega_{\alpha,\sigma}(\delta_\ell)}{g_{\alpha,\sigma}(\delta_\ell)} \cdot q^{2\sigma \theta(\ell+\kappa)  } \cdot  p_{t-s}\(\ell,\kappa\)\nn \\
\eeq
where we use the notation $p_t(\kappa,\ell)$ for the one-dual particle transition probability from site $\kappa$ to site $\ell$ at time $t$.
%
%
%
 \noindent
The interest of  the correlations in \eqref{covv} lies in the link between the function $N^-_i(x)$ and the total current at site $i$ as shown in the following definition and proposition (see also section 6.2 of \cite{CGRS1} for a more detailed treatment of the subject).
 \vskip.3cm
 \bd[Current] Let $\{x(s), \; s \ge 0\}$ be a c\`adl\`ag trajectory on $\Omega_L$, then
the total integrated current $J_i(t)$ in the time interval $[0,t]$ is defined as the net number of particles crossing the bond $(i-1,i)$
in the left direction. Namely, let $(t_k)_{k\in\mathbb{N}}$ be the sequence of the process jump times. Then
\be
J_i(t) = \sum_{k: t_k \in [0,t]}  \(\mathbf{1}_{\{x(t_k) = x(t_k^-)^{i,i-1}\}} - \mathbf{1}_{\{x(t_k) = x(t_k^-)^{i-1,i}\}}\),\qquad i\in \Lambda_L
\ee
\ed
\bl
The total integrated current of a c\`adl\`ag trajectory $\{x(s), \; 0\le s\le t\}$ with $x(0)=x$ is given by
\be \label{J}
J_i(t)= N^-_i(x(t))-N^-_i(x), \qquad i\in \Lambda_L.
\ee
\el
\bpr
\eqref{J} immediately follows from the definition of $J_i(t)$.
\epr
\vskip.4cm
\noindent
As a consequence of \eqref{J} we have that the duality relation gives information about the $q^{-2}$-exponential moments and correlations of the currents. 
The convenient  use of duality for  the computation of $q$-exponential moments of the current has already emerged in \cite{CGRS1}. Here  the authors pointed out the link between these moments and the triangular self-duality function for the case of ASIP$(q,\theta)$. Thanks to this link an explicit formula was found for the expectation of the observable $q^{2J_i(t)}$ when the process is initialized from a deterministic configuration $\eta$. The added value of the orthogonal polynomial duality functions lies in the possibility to compute the two-times correlations of the type \eqref{covv} by a relatively simple computation. 
\vskip.3cm
\noindent
The form of $q$- Krawtchouk  and $q$-Meixner polynomials suggests  that the duality relation \eqref{DuRel} is amenable to provide informations about all the $q^{-2}$-exponential moments of the variables $N^-_i(x)$, i.e.
 \be\label{q-exp}
 \mathbb E_x \left[\prod_{k=1}^{K} q^{-2m_kN^-_{i_k}(x(t))} \right], \qquad K\in \N, \quad m_k\in \N, \quad 1\le i_1<\cdots<i_K\le L
 \ee
 and we expect that formulas of the type of \eqref{covv} for the stationary space-time correlations can be obtained for any polynomials in the variable $q^{-2N^{-}_{i}(\cdot)}$, $i\in \Lambda_L$ by direct computation of the scalar product in \eqref{Cf}. Computation of moments of this type will be object of future investigation.

\section{Construction of the orthogonal dualities} \label{proofofresults}
From the analysis  developed in \cite{CGRS1} and \cite{CGRS}   the processes ASIP($q,\theta$) and ASEP($q,\theta$) are known to be self-dual with respect to self-duality functions that have a nested-product structure and  a {\em triangular} form, with triangular  meaning that they have support contained in the set of couples $(x,n)\in \Omega_L^2$ such that $n_i\le x_i$ for all $i\in \Lambda_L$.
In this section  we start   from these triangular duality-functions to construct new duality functions satisfying  suitable orthogonality relations.

\subsection{Triangular dualities}  \label{ssec:triangular dualities}

The functions
\be
\label{tria1asip}
D_{\lambda}^{\text{tr}}(x, n )= {\prod_{i=1}^L}  \frac{\binom{x_i}{n_i}_q}{\Psi_{q,\sigma}(\theta,n_i)}\,
\cdot \,
q^{x_i\left(2N^-_{i-1}(n) +n_i\right)- 2\sigma\theta  i n_i} \, \lambda^{n_i}
\ee
and
\be
\label{tria2asip}
 \widehat D_{\lambda}^{\text{tr}}( x, n )=
 {\prod_{i=1}^L} \frac{\binom{x_i}{n_i}_q}{\Psi_{q,\sigma}(\theta,n_i)}\,
\cdot \,
q^{-n_i\left(2N^-_{i-1}(x) +x_i\right)- 2\sigma\theta  i n_i} \, \lambda^{n_i}
\ee
with $\Psi$ the $q$-binomial coefficient given in \eqref{Psi}, are self-duality functions for the ASEP($q,\theta$) (for $\sigma=-1$), resp.~for ASIP($q,\theta$) (for $\sigma=+1$).
For the proof of the duality relation we refer to  \cite[Theorem 5.1]{CGRS1} for the case  $\sigma=+1$ and to  \cite[Theorem 3.2]{CGRS} for $\sigma=-1$. We notice moreover that these two functions are the same function modulo a multiplicative quantity that only depends on the total number of particles $N(x)$ and $N(n)$. {More precisely, using \eqref{N1}}, we have that
\beq
\widehat D_{\lambda}^{\text{tr}}( x, n )= q^{-2N(n)N(x)} \cdot  D_{\lambda}^{\text{tr}}( x, n ).
\eeq

\subsection{From triangular to orthogonal dualities} \label{generalresult}
The following theorem, {which is a slight generalization of \cite[Proposition 4.5]{CFGGR},} will be the key ingredient needed to produce biorthogonal duality functions from the triangular ones.

\begin{theorem}[Biorthogonal self-duality functions via scalar product] \label{prop:orthogonality}
Let $X$ be a Markov process on a countable state space $\Omega$, with generator $L$. Let $\mu_1$ and $\mu_2$ be two reversible measures for $X$,
and  $d_1$, $d_2$, $\tilde d_1$ and  $\tilde d_2$ be four self-duality functions for $X$. Suppose that
\begin{equation}\label{d12}
\langle d_1(x,\cdot), {d}_2(\cdot,n) \rangle_{\mu_1}=\dfrac{\delta_{x,n}}{\mu_2(n)} \;\qquad\text{and}\qquad\langle \tilde{d}_2(x,\cdot), \tilde{d}_1(\cdot,n) \rangle_{\mu_2}=\dfrac{\delta_{x,n}}{\mu_1(n)}
\end{equation}
for $x,n \in \Omega$. Here $\langle \, \cdot\, , \, \cdot\, \rangle_{\mu_i}$ denotes the scalar product corresponding to the measure $\mu_i$. Then the functions $D, \widetilde D : \Omega \times \Omega \to \R$ given by
\be \label{ordu}
{ D}(x,n) := \langle \tilde{d}_1(x,\cdot),  d_1(n,\cdot)\rangle_{ \mu_1} \qquad \widetilde{D}(x,n) := \langle \tilde{d}_2(\cdot,x) , d_2(\cdot,n) \rangle_{\mu_1}
\ee
are self-duality functions for $X$. Moreover, they satisfy the biorthogonality relations
\be \label{or}
\langle D(\cdot,m), \widetilde{D} (\cdot,n)  \rangle_{\mu_2} = \frac{ \delta_{m,n}}{\mu_2(n)}, \qquad m,n \in \Omega.
\ee
In particular, if $\widetilde D = c_1(x)c_2(n) D$, where $c_1$ (resp. $c_2$) is a positive function of the total number of particles (resp. dual particles), then
equation \eqref{or} becomes an orthogonality relation for $D$ with respect to the weight $c_1(x)\mu_2(x)$ and with squared norm $ \frac{ 1}{c_2(n) \mu_2(n)}$.

\end{theorem}

\bpr
Since scalar products of self-duality functions are self-duality function by \cite[Proposition 4.1]{CFGGR}, we have that both $D$ and $\widetilde{D}$ are self-duality functions. For the biorthogonality relation, assuming we can interchange the order of summation, we get
\begin{align*}
\langle D(\cdot,m), \widetilde{ D}(\cdot,n)\rangle_{\mu_2} &= \sum_{x} D(x,m) \widetilde{ D}(x,n) \mu_2(x) \\
& = \sum_{x} \left(\sum_y \tilde{d}_1(x,y)d_1(m,y)  \mu_1(y)\right) \left( \sum_z \tilde{d}_2(z,x)  d_2(z,n) \mu_1(z)\right)\mu_2(x)\\
& = \sum_{y,z} { d}_1(m,y) d_2(z,n)  \mu_1(y)\mu_1(z)\sum_x \tilde{d}_2(z,x)  \tilde{d}_1(x,y)\mu_2(x) \\
& = \sum_{y,z}  \mu_1(y)\mu_1(z) { d}_1(m,y) d_2(z,n) \, \frac{ \delta_{y,z}}{\mu_1(y)} \\
& = \sum_y { d}_1(m,y) d_2(y,n) \mu_1(y) = \frac{ \delta_{m,n}}{\mu_2(m)}.
\end{align*}
This proves the result.
\epr

\vskip.5cm
\noindent
In order  to apply this theorem to produce biorthogonal self-duality functions from the triangular ones we need to show that the triangular duality functions \eqref{tria1asip} and \eqref{tria2asip} satisfy the relations \eqref{d12}.  This property is the content of  proposition below.

\vskip.2cm
\noindent
Let $\mu_\alpha$, $\alpha\in\R\setminus\{0\}$ be the family of reversible signed measures defined in \eqref{SM}, then from now onward we will use the notation
 $\langle \, \cdot\, , \, \cdot\, \rangle_\alpha$ for the scalar product with respect to the reversible measure $\mu_\alpha$.

\bp\label{Prop1} Let $D_{\lambda}^{\textup{tr}}$ and $\widehat D_{\lambda}^{\textup{tr}}$ the functions defined in \eqref{tria1asip}-\eqref{tria2asip}, then, for all $\alpha,\beta\in \R\setminus\{0\}$  we have
 \beq
 &&\langle D_{1/{\alpha q}}^{\textup{tr}}(x,\cdot), \widehat D_{-q/\beta}^{\textup{tr}}(\cdot,n)\rangle_{-\alpha}    = \frac{\delta_{x,n}}{\mu_\beta(n)} \;.
 \eeq
\ep
\noindent
We will prove this result  in Section \ref{p5.1} only for ASEP($q,\theta$) as the proof for ASIP($q,\theta$) is similar.

\vskip.4cm
\noindent
Proposition \ref{Prop1} guarantees that the two conditions in \eqref{d12} are satisfied for the self duality functions
\[
d_{1}= D^{\text{tr}}_{1/\alpha q}, \quad \tilde d_{1}= \widehat  D^{\text{tr}}_{q/\alpha}, \quad
\quad d_{2}= \widehat D^{\text{tr}}_{-q/\beta}, \quad \tilde d_{2}= D^{\text{tr}}_{-1/\beta q} \]
by taking the scalar product with respect to the measures
\[
 \quad \mu_1=\mu_{-\alpha}, \quad  \mu_2=\mu_\beta.
\]
then, as a consequence of Theorem \ref{prop:orthogonality}, we can deduce that the functions
\beq
D_\alpha(x,n) &:=\langle \widehat D^{\text{tr}}_{q/\alpha}(x,\cdot), D^{\text{tr}}_{1/\alpha q}(n,\cdot)\rangle_{-\alpha},\label{qua} \\
\widetilde D_{\alpha,\beta}(x,n)&:=\langle   D^{\text{tr}}_{-1/\beta q}(\cdot , x),\widehat D^{\text{tr}}_{-q/\beta} (\cdot , n)\rangle_{-\alpha},\label{quaa}
\eeq
are again self-duality functions satisfying the following
 biorthogonality  relation:
 \be\label{ddDD}
\langle D_\alpha(\cdot , m),\widetilde D_{\alpha,\beta}(\cdot , n)\rangle_{\mu_\beta}=\dfrac{\delta_{n,m}}{\mu_{\beta}(n)}.
\ee

\vskip.4cm
\noindent
{\bf {\large Conclusion of the proof for ASEP($q,\theta$).}} The next step in the construction of the orthogonal dualities  is the computation of the  explicit expressions for the self-duality functions $D_\alpha$ and $\widetilde D_{\alpha,\beta}$ that have been implicitly defined in \eqref{qua}-\eqref{quaa}. This is the content of the next proposition  where the new duality functions are identified, for the case $\sigma=-1$,   in terms of $q-$Krawtchouk polynomials.

\bp \label{propasep}
Let $\sigma=-1$,  $\alpha,\beta\in \R\setminus \{0\}$, then the functions $D_\alpha(x,n)$ and $\widetilde D_{\alpha,\beta}(x,n)$  defined in \eqref{qua}-\eqref{quaa} are given by
\beq \label{Da}
&&D_\alpha(x,n) :=  \prod_{i =1}^L  K_{n_i}(q^{-2x_i};p_{i,\alpha}(x,n),\theta ;q^2), \nn\\
&& p_{i,\alpha}(x,n):=\frac 1\alpha \; q^{-2(N_{i-1}^-(x)-N^+_{i+1}(n))+\theta(2i-1)-1}
\eeq
and
\beq\label{Dt}
\begin{split}
\widetilde D_{\alpha,\beta}(x,n):=\frac{ (\alpha q^{1+2N(x) - 2\theta L-\theta})_\infty \; \;q^{N(x)(N(x)-1)} }{ (\alpha q^{1-2N(n) - \theta})_\infty \; \;q^{N(n)(N(n)-1)} }   \left(\dfrac{\alpha}{ \beta} \right) ^{N(x+n)}\cdot D_\alpha(x,n).
\end{split}
\eeq

\ep
\noindent
Proposition \ref{propasep} will be proved in Section \ref{p4.2}. The  function $D_\alpha$ emerging here is nothing else than the self-duality function $D_\alpha^{\text{\tiny{ASEP($q,\theta$)}}}$ defined in Theorem \ref{teoasep}.
 Whereas $\widetilde D_{\alpha,\beta}$ is another self-duality function differing from $ D_{\alpha}$ only via multiplication by a factor that depends only on the total number of particles in both configurations, $N(x)$ and $N(n)$. To conclude the proof of Theorem \ref{teoasep} it remains to turn the biorthogonality relation \eqref{dd} in an orthogonality relation for $D_\alpha$. This is possible by including the extra factor in \eqref{Dt} in the  measure with respect to which we take the scalar product. So, at this point Theorem \ref{teoasep} follows from Theorem \ref{prop:orthogonality}, \eqref{dd} and Proposition \ref{propasep} after choosing $\alpha=\beta$ and switching from the scalar product with respect to $\mu_\alpha$ to the scalar product with respect to $\omega_\alpha$ (defined in \eqref{w}).
\vskip.5cm
\noindent
{\bf {\large Conclusion of the proof for ASIP($q,\theta$).}} The strategy followed for the case $\sigma=-1$ does not completely work for $\sigma=1$.  In this case Theorem \ref{prop:orthogonality} can  only be partially applied. More precisely we have that
the scalar product \eqref{quaa} formally defining $\widetilde D_{\alpha,\beta}$ does not converge, as it gives rise, now, to an infinite sum.
Nevertheless we have that  the hypothesis \eqref{d12} are satisfied as Proposition \ref{Prop1} holds true also for $\sigma=1$ and the scalar product \eqref{qua} defining $D_\alpha$ converges. The explicit computation of this scalar product  gives rise to the multivariate $q-$Meixner polynomials $D_\alpha^{\text{\tiny{ASIP($q,\theta$)}}}$ defined in \eqref{D}. This is, due to Theorem \ref{prop:orthogonality} a self-duality function.  It remains to prove, a posteriori, an orthogonality relation that can be guessed exploiting   the formal similarities between ASIP and ASEP. The proof of this orthogonality relation will be the object of  Section \ref{ASIP}.
\vskip.5cm
\noindent
Before entering the details of our proofs, one may wonder if there is a link from the orthogonal dualities to the triangular ones. In the symmetric case this has been revealed in Remark 4.2 of \cite{RS} where the authors show that, after a proper normalization, as $\alpha \rightarrow 0$ the orthogonal dualities are precisely the triangular ones. 
A similar result holds true in the asymmetric context, however, the outcome of the limit is the triangular duality up to a factor that depends on the total number of (dual) particles, namely $q^{N(n)^2}\, \widehat D^{\mathrm{tr}}_1(x,n)$ or $q^{-N(x)^2}\, D^{\mathrm{tr}}_1(n,x)$  (depending on which variable we assume bigger).
The constant factor converges to $1$ as soon as $q\rightarrow 1$, see Remark \ref{frank} in the next Section.
}

\section{Proofs for ASEP$(q,\theta)$}\label{proofs}

In this section we will prove Theorem \ref{teoasep}.
In the proofs it will be convenient to write the triangular duality functions given in Section \ref{ssec:triangular dualities} as nested products of ``1-site duality functions". Let $\lambda,p,r \in \R\setminus\{0\}$. We define for $n,k\in S_{\sigma,\theta}$,
\[
\begin{split}
{d}_{\lambda}(n,k; p, r) &:=  \frac{\binom{n}{k}_q}{\Psi_{q,\sigma}(\theta,k)}\,\lambda^{k}\, q^{n k} \, p^{n} \, r^{k}\, \mathbf 1_{k \le n}, \\
\hat d_{\lambda}(n,k; p, r) &:=  \frac{\binom{n}{k}_q}{\Psi_{q,\sigma}(\theta,k)}\,\lambda^{k}\, q^{- n k} \, p^{-k} \, r^{k}\, \mathbf 1_{k \le n}.
\end{split}
\]
Then the triangular duality functions are given by
\be
\label{triangular duality product}
D_{\lambda}^{\text{tr}}(x, n) =
 \prod_{i =1}^L d_\lambda(x_{i}, n_{i}; p_{i}, r_{i}), \qquad \widehat D_{\lambda}^{\text{tr}}( x, n )=
 \prod_{i =1}^L \hat{d}_{\lambda}(x_{i}, n_{i}; \hat p_{i}, r_{i}),
\ee
where
\[
p_{i}=p_i(n)=q^{2N^{-}_{i-1}(n)}, \quad \hat p_i= \hat p_i(x) = q^{2N^-_{i-1}(x)}, \qquad  r_{i}=q^{-2i\sigma\theta}.
\]
Note that the \emph{nested} product structure comes only from the parameters $p_i$ and $\hat p_i$.

Furthermore, recall that for both processes we have families of reversible measures labelled by $\alpha$.

\subsection{Proof of Proposition \ref{Prop1} {for ASEP($q,\theta$)}}\label{p5.1}
In order to prove Proposition \ref{Prop1} we start by writing the scalar product with free parameters $\lambda_{1}$ for $D^{\text{tr}}$ and $\lambda_2$ for $\widehat D^{\text{tr}}$ and throughout the computation the right choice will become clear. We have
\begin{align} \label{inizio}
	\langle D_{\lambda_1}^{\text{tr}}(x,\cdot), \widehat D_{\lambda_2}^{\text{tr}}(\cdot,n)\rangle_{-\alpha} = \prod_{i=1}^{L} \sum_{y_i = n_i}^{x_i} d_{\lambda_1}(x_{i},y_i, p_i, r_i)  \hat d_{\lambda_2}(y_i,n_{i}; \hat p_i, r_i ) \mu_{-\alpha}(y_i),
\end{align}
where both  $p_i$ and $\hat p_i$ depend on $N^-_{i-1}(y)$, making the display above a nested product of sums. Since the sum over $y_i$ depends on $y_1,\ldots,y_{i-1}$ we first evaluate the sum over $y_L$, then the sum over $y_{L-1}$, and so on. Let us denote the sum over $y_i$ by $\Sigma_i(x_i,n_i;y)$, where $y=(y_1,\ldots,y_{i-1})$ (we suppress the dependence on $r_i, \lambda_1, \lambda_2$ and $\alpha$), then 
\beq
&\Sigma_i(x_i,n_i;y)&=\sum_{y_i = n_i}^{x_i} d_{\lambda_1}(x_{i},y_i, p_i, r_i)  \hat d_{\lambda_2}(y_i,n_{i}; \hat p_i, r_i ) \mu_{-\alpha}(y_i)\nn\\
&&=\sum_{y_i=n_i}^{x_i} \frac{\binom{x_i}{y_i}_q}{\binom{\theta}{y_i}_q}
q^{x_i\left[2 N^{-}_{i-1}(y) +y_i\right]+ 2\theta  i y_i} \, \lambda_1^{y_i}  \frac{\binom{y_i}{n_i}_q}{\binom{\theta}{n_i}_q} 
\, q^{-n_i\left[2N_{i-1}^{-}(y) +y_i\right]+ 2\theta  i n_i} \, \lambda_2^{n_i} \cdot \nn \\
&&\qquad \cdot {(-\alpha)^{y_i}} \,{\binom{\theta}{y_i}_q}  q^{-2\theta iy_i}  \nn \\
&&= {q^{2\theta  in_i} \lambda_2^{n_i}}  \sum_{y_i=n_i}^{x_i}  \frac{\binom{y_i}{n_i}_q \binom{x_i}{y_i}_q}{\binom{\theta}{n_i}_q}
q^{(x_i-n_i)\left[2N^{-}_{i-1}(y) +y_i\right]} \, \(-\alpha\lambda_1\)^{y_i}  \nn \\
&&= \frac{\binom{x_i}{n_i}_q}{\binom{\theta}{n_i}_q} {q^{2\theta in_i} \lambda_2^{ n_i}}
q^{(x_i-n_i)\left[2N^{-}_{i-1}(y)\right] }
\sum_{y_i=n_i}^{x_i} {\binom{x_i-n_i}{y_i-n_i}_q}
q^{(x_i-n_i)y_i}\,\(-\alpha\lambda_1\)^{y_i}, \nn
\eeq
where in the last equality we used the $q-$binomial identity \eqref{qbinid}.
Performing a change of variables in the summation and setting
\begin{equation*}
	C_i(x_i,n_i;y):= \frac{\binom{x_i}{n_i}_q}{\binom{\theta}{n_i}_q}\;{q^{2\theta in_i} \lambda_2^{ n_i}} (-\alpha\lambda_1 )^{n_i} q^{(x_i-n_i)(2N^{-}_{i-1}(y) +n_i)} \;,
\end{equation*}
we get
\begin{equation*}
	\Sigma_i(x_i,n_i;y) =   C_i(x_i,n_i;y)  \sum_{z=0}^{x_i-n_i} \binom{x_i-n_i}{z}_q \cdot
	(-\alpha\lambda_1)^{z}  q^{(x_i-n_i)z} \;.
\end{equation*}
Then the Newton formula in equation \eqref{newton} yields
\begin{equation*}
	\Sigma_i(x_i,n_i;y)=  C_i(x_i,n_i;y) (\alpha \lambda_1 q)_{x_i-n_i}\;.
\end{equation*}
First let us choose $\lambda_1=\frac {q^{-1}} {\alpha}$, then the product is non-zero only for $x_i=n_i$,
\begin{equation*}
	\Sigma_i(x_i,n_i;y)= C_i(n_i,n_i;y)\, \delta_{x_i,n_i} \;,
\end{equation*}
where it should be remarked that $C_i(n_i,n_i;y)$ is independent of $y=(y_1,\ldots,y_{i-1})$.
Next choosing $\lambda_2=-\frac q \beta$ we find
\begin{equation*}
	C_i(n_i,n_i)=\frac{q^{2\theta in_i}}{\binom{\theta}{n_i}_q}\, \(-\frac{\lambda_2}{q}\)^{n_i} = \frac 1 {\mu_\beta(n_i)} \;.
\end{equation*}
Using this in equation \eqref{inizio}, we get
\begin{equation*}
	\langle D_{1/\alpha q}^{\text{tr}}(x,\cdot), \widehat D_{-q/\beta}^{\text{tr}}(\cdot,n)\rangle_{-\alpha}=\prod_{i=1}^{L}   \frac{\delta_{x_i,n_i}}{\mu_\beta(n_i)}      = \frac{\delta_{x,n}}{\mu_\beta(n)} \;,
\end{equation*}
which concludes the proof of the proposition.
\qed

\subsection{Proof of Proposition \ref{propasep}}\label{p4.2}
The explicit expressions will follow from calculations involving $q-$binomials coefficients and $q-$hypergeometric functions. We start with the biorthogonality property.
\vskip.4cm
\noindent
\textbf{Calculation of $D$.}
We fix $x,n \in \Omega_L$, {$\alpha >0$}, and we evaluate
\[
D_\alpha(x,n)=\langle \widehat D^{\text{tr}}_{q/\alpha}(x,\cdot), D^{\text{tr}}_{1/\alpha q}(n,\cdot)\rangle_{-\alpha}.
\]
We make use of the product structure \eqref{triangular duality product} again. We start with a result for the 1-site duality functions.
\begin{lemma} \label{lem:scalar product 1-site}
For $p,\hat p\in \R\setminus\{0\}$, $m \in \Z$ and $s,t \in \{0,\ldots,\theta\}$,
\[
\begin{split}
\sum_{y =0}^{\theta} q^{2my} \,& \hat d_{q/\alpha}(s,y; \hat p,r_i) d_{1/\alpha q}(t,y; p, r_i) \mu_{-\alpha}(y)  = \\
& p^t
\rphis{2}{1}{q^{-2s},q^{-2t}}{q^{-2\theta }}{q^2, \frac{1}{\alpha \hat p} q^{1+ 2t - \theta+2i\theta +2m} }.
\end{split}
\]
\end{lemma}
\begin{proof}
Using the explicit expressions of the 1-site duality functions we find
\[
\begin{split}
\sum_{y =0}^{\theta}& q^{2my}\, \hat d_{q/\alpha}(s,y; \hat p,r_i) d_{1/\alpha q}(t,y; p, r_i) \mu_{-\alpha}(y) \\
& =\sum_{y \le s \wedge t }
 \frac{\binom{s}{y}_q  \binom{t}{y}_q}{\binom{\theta}{y}_q \binom{\theta}{y}_q}\,
q^{-sy +4i\theta y + ty +2my }   p^t \,   \left(   \dfrac{1}{ \alpha \alpha \hat p}\right) ^{y}
{(-\alpha)^{y}} \,{\binom{\theta}{y}_q} q^{-2i\theta y} \\
&=p^t \sum_{y \le s \wedge t } \dfrac{\binom{s}{y}_q  \binom{t}{y}_q }{\binom{\theta}{y}_q}
\left( -\dfrac{1}{ \alpha \hat p}\right) ^{y} q^{y\left(t -s +2m +2i\theta  \right) } \\
& = p^t \sum_{y \le s \wedge t} \frac{ (q^{-2s})_y (q^{-2t})_y }{(q^2)_y (q^{-2\theta })_y}  \left( \frac{1}{\alpha \hat p} q^{1+ 2t - \theta+2i\theta +2m} \right)^y \;,
\end{split}
\]
where the last equality is due to the $q-$binomial coefficient identity \eqref{qbinom2}. The result then follows from the definition of the $_2\varphi_1$-function.
\end{proof}

\noindent
We introduce auxiliary functions: for $i=1,\ldots,L$,
\[
A_i(y_1,\ldots,y_i) = \hat{ d}_{q/\alpha}(x_i,y_i; \hat p_i(x),r_i) d_{1/\alpha q}(n_i, y_i; p_i(y), r_i) \mu_{-\alpha}(y_i).
\]
From Lemma \ref{lem:scalar product 1-site} with
\[
s=x_i, \quad t=n_i, \quad p=p_i(y) = q^{2N_{i-1}^-(y)}, \quad \hat p = \hat p(x) = q^{2N_{i-1}^-(x)},\quad m = n_{i+1},
\]
we find the following identities.
\begin{lemma}  \label{lem:sum A}
Let $i \in \{1,\ldots,L\}$ and $y_1,\ldots,y_{i-1} \in \{0,\ldots,\theta\}$, then
\[
\sum_{y_i} A_i(y_1,\ldots,y_i) q^{2n_{i+1}N_{i}^-(y)} = \sum_{y_i} S_{i}(y_i;x,n) q^{2n_{i+1}N_{i}^-(y)+2n_{i}N_{i-1}^-(y)},
\]
where $n_{L+1}=0$, and
\[
S_i(y_i;x,n) = \frac{ (q^{-2x_i})_{y_i} (q^{-2n_i})_{y_i} }{(q^2)_{y_i} (q^{-2\theta })_{y_i} } \left(\frac{1}{\alpha} q^{1-2N_{i-1}^-(x)+2n_i-\theta+2i\theta }\right)^{y_i}.
\]
\end{lemma}
\noindent
Now we are ready to find an explicit expression for $D(x,n)$.
We have
\[
\begin{split}
D(x,n) &=\langle \widehat D^{\text{tr}}_{q/\alpha}(x,\cdot), D^{\text{tr}}_{1/\alpha q}(n,\cdot)\rangle_{-\alpha}\\
& = \sum_{y_1} A_1(y_1) \sum_{y_2} A_2(y_1,y_2) \ldots \sum_{y_L} A_L(y_1,\ldots,y_L).
\end{split}
\]
From induction, using Lemma \ref{lem:sum A},  we obtain
\[
D(x,n) = \sum_{y} q^{2\sum_{i=1}^L n_i N_{i-1}^-(y)}\prod_{i=1}^L S_{i}(y_i;x,n).
\]
We apply identity \eqref{N2}, then
\[
\begin{split}
D(x,n)& = \sum_{y} q^{2\sum_{i=1}^L y_i N_{i+1}^+(n)}\prod_{i=1}^L S_{i}(y_i;x,n) \\
&= \prod_{i=1}^L\sum_{y_i} S_{i}(y_i;x,n) q^{2y_iN_{i+1}^+(n)}.
\end{split}
\]
Finally, using the explicit expression for $S$ and the definition \eqref{eq:2phi1} of the $_2\varphi_1$-function we find
\[
\begin{split}
D(x,n)&= \prod_{i=1}^L\sum_{y_i=0}^{x_i \wedge n_i} \frac{ (q^{-2x_i})_{y_i} (q^{-2n_i})_{y_i} }{ (q^2)_{y_i} (q^{-2\theta })_{y_i} }\; \alpha^{-y_i} q^{y_i(1+2n_i+2N_{i+1}^+(n)-2N_{i-1}^-(x)-\theta+2i\theta )} \\
& =\prod_{i=1}^L \rphis{2}{1}{q^{-2x_i}, q^{-2n_i}}{q^{-2\theta }}{q^2, \frac{1}{\alpha}q^{1+2n_i+2N_{i+1}^+(n)-2N_{i-1}^-(x)-\theta+2i\theta } }.
\end{split}
\]
Comparing this with the definition of the $q-$Krawtchouk polynomials \eqref{def:q-Krawtchouk}, we see that $D(x,n)$ is indeed a nested product of $q-$Krawtchouk polynomials. \\

\begin{remark}[From orthogonal dualities to triangular dualities] \label{frank}
The triangular duality functions can be recoved from the duality function $D(x,n)$ by taking an appropriate limit. Indeed, note that the $_2\varphi_1$-function is a polynomial in $\alpha^{-1}$ of degree $x_i \wedge n_i$. Assuming $n_i\leq x_i$ it follows that 
\[
\begin{split}
	\lim_{\alpha \to 0} &(-\alpha)^{n_i} \rphis{2}{1}{q^{-2x_i}, q^{-2n_i}}{q^{-2\theta }}{q^2, \frac{1}{\alpha}q^{1+2n_i+2N_{i+1}^+(n)-2N_{i-1}^-(x)-\theta+2i\theta } } \\
	& = 	\frac{ (q^{-2x_i})_{n_i} }{(q^{-2\theta})_{n_i}} q^{-2n_iN_{i-1}^-(x) +n_i(2i-1)\theta +2n_iN_{i+1}^+(n)+n_i^2}.
\end{split}
\]
Comparing this with the 1-site duality function $\hat d_1(x_i,n_i;\hat p_i,r_i)$ defined in the beginning of this section and the definition \eqref{triangular duality product} of the triangular duality function, we obtain
\be\label{trort}
\begin{split}
\lim_{\alpha \to 0} (-\alpha)^{N(n)} D_{\alpha}(x,n) & = \prod_{i=1}^L \hat d_1(x_i,n_i, q^{2N_{i-1}^-(x)}, q^{2i\theta})\, q^{2n_i N_{i+1}^+(n) +n_i^2}\\
& = q^{N(n)^2}\, \widehat D^{\mathrm{tr}}_1(x,n),
\end{split}
\ee
assuming $n_i\leq x_i$ for $i=1,\ldots,n$. Here we used identity \eqref{N1} as well as $\sum_{i=1}^L x_i N_{i-1}^-(x) = \sum_{i=1}^L x_i N_{i+1}^+(x)$.
\newline
Similarly, for $x_i \leq n_i$ we obtain
\[
\begin{split}
	\lim_{\alpha \to 0} &(-\alpha)^{x_i} \rphis{2}{1}{q^{-2x_i}, q^{-2n_i}}{q^{-2\theta }}{q^2, \frac{1}{\alpha}q^{1+2n_i+2N_{i+1}^+(n)-2N_{i-1}^-(x)-\theta+2i\theta } } \\
	& = 	\frac{ (q^{-2n_i})_{x_i} }{(q^{-2\theta})_{x_i}} q^{2x_iN_{i+1}^+(n) +x_i(2i-1)\theta+2x_in_i -2x_iN_{i-1}^-(x)-x_i^2}.
\end{split}
\]
Comparing this with the 1-site duality function $d_1(n_i,x_i;p_i,r_i)$ and the corresponding triangular duality function it follows that
\[
\begin{split}
	\lim_{\alpha \to 0} (-\alpha)^{N(x)} D_{\alpha}(x,n)  &= \prod_{i=1}^L d_1(n_i,x_i;q^{2N_{i-1}^-(n)}, q^{2i\theta}) \, q^{-2 x_i N_{i-1}^-(x)-x_i^2} \\
	& = q^{-N(x)^2}\, D^{\mathrm{tr}}_1(n,x),
\end{split}
\]
provided $x_i\leq n_i$ for $i=1,\ldots,n$. 
\end{remark}

\vskip.5cm
\noindent
\textbf{Calculation of $ \widetilde D$.}
The calculation of $\widetilde D$ is similar to the calculation for $D(x,n)$, but a bit more involved. We fix $x,n \in \Lambda_L$, {$\alpha >0$}, and we evaluate
\[
\widetilde D(x,n)=\langle   D^{\text{tr}}_{-1/\beta q}(\cdot , x),\widehat D^{\text{tr}}_{-q/\beta} (\cdot , n)\rangle_{-\alpha},
\]
for some $\beta \in \R$.
We start with a result for $1$-site duality functions again.
\begin{lemma} \label{lem2:scalar product 1-site}
For $\beta,p,\hat p \in \R\setminus\{0\}$, $m\in \N$ and $s,t \in \{0,\ldots,\theta\}$,
\[
\begin{split}
\sum_{y =0}^{\theta} &q^{-2my} d_{-1/\beta q}(y,s;p,r_i) \hat d_{-q/\beta}(y,t;\hat p, r_i) \mu_{-\alpha}(y) = \\
&\left(\frac{\alpha}{\beta} \right)^{s+t} p^{s+t} \hat p^{-t} q^{(s+t)(1+s-t-2m)} q^{-2s} \frac{(\alpha p q^{1+2s-2m-2i\theta -\theta})_{\infty}}{(\alpha p q^{1-2t-2m-2i\theta +\theta})_\infty} \\
& \quad \cdot \rphis{2}{1}{q^{-2s}, q^{-2t} }{q^{-2\theta }}{q^2, \frac{1}{\alpha p} q^{1+2t+2i\theta -\theta+2m}}.
\end{split}
\]
\end{lemma}
\begin{proof}
Let us denote the sum on the left hand side by $\Sigma$. From the explicit expressions of the 1-site duality functions we find
\[
\begin{split}
\Sigma &= \sum_{y=0}^{\theta} q^{-2my} d_{-1/\beta q}(y,s;p,r_i) \hat d_{-q/\beta}(y,t;\hat p, r_i) \mu_{-\alpha}(y) \\
& = \sum_{y \geq s \vee t} \frac{ \binom{y}{s}_q \binom{y}{t}_q }{ \binom{\theta}{s}_q \binom{\theta}{t}_q } (-\beta q)^{-s} (-q/\beta)^t q^{y(s-t-2m)} p^y \hat p^{-t} q^{2i\theta (s+t)} \binom{\theta}{y}_q (-\alpha)^y q^{-2i\theta y} \\
& = C_1 \sum_{y \geq s \vee t} \binom{y}{s}_q \binom{y}{t}_q \binom{\theta}{y}_q C_2^y,
\end{split}
\]
where $C_2 = -\alpha p q^{s-t-2i\theta -2m}$ and
\[
\begin{split}
C_1 & = \frac{ (-\beta q^{-2i\theta } )^{-(s+t)} q^{t-s} \hat p^{-t} }{\binom{\theta}{s}_q \binom{\theta}{t}_q} \\
& = (\beta q^{-2i\theta  +2 +\theta})^{-s} (\beta \hat p q^{-2i\theta  +\theta} )^{-t} \frac{(q^2)_s (q^2)_t}{(q^{-2\theta })_s (q^{-2\theta })_t}.
\end{split}
\]
We focus on the sum. Assume $s \leq t$ and let $C$ be an arbitrary constant, then we obtain from Lemma \ref{applemma},
\[
\begin{split}
S &:= \sum_{y \geq s \vee t} \binom{y}{s}_q \binom{y}{t}_q \binom{\theta}{y}_q  C^y  \\
&  = \frac{ (-C)^t q^{t(1+\theta-s)}q^{s^2} (q^{-2\theta })_t}{(q^2)_s (q^2)_{t-s} } \rphis{2}{1}{q^{2t+2}, q^{2t-2\theta } }{q^{2+2t-2s}}{q^2, -Cq^{1-s-t+\theta}}.
\end{split}
\]
Next we transform this $_2\varphi_1$-series into another $_2\varphi_1$-series using Heine's transformation \eqref{euler}, and then we reverse the order of summation, see identity \eqref{change}, to obtain
\[
\begin{split}
&\rphis{2}{1}{q^{2t+2}, q^{2t-2\theta }}{q^{2+2t-2s}}{q^2, -Cq^{1-s-t+\theta}}  \\
& \qquad = \frac{(-Cq^{1+s+t-\theta})_{\infty}}{(-Cq^{1-s-t+\theta})_\infty}
 \rphis{2}{1}{q^{-2s}, q^{2-2s+2\theta } }{q^{2+2t-2s}}{q^2, -Cq^{1+s+t-\theta}} \\
& \qquad =  (Cq^{1+s+t-\theta})^s q^{-s-s^2} \frac{ (q^{2+2\theta -2s})_s }{ (q^{2+2t-2s})_s} \frac{(-Cq^{1+s+t-\theta})_{\infty}}{(-Cq^{1-s-t+\theta})_\infty}\\
 &  \qquad \qquad \cdot \rphis{2}{1}{q^{-2s}, q^{-2t} }{q^{-2\theta }}{q^2, -\frac{1}{C} q^{1+s+t-\theta}} \;.
\end{split}
\]
Using identities \eqref{1.8.15koeoek} and \eqref{eq:id Pochhammer} for the $q-$Pochhammer symbols this gives us
\[
S = (-Cq^{\theta+1})^{s+t} \frac{ (q^{-2\theta })_s (q^{-2\theta })_t  }{ (q^2)_s (q^2)_t } \frac{(-Cq^{1+s+t-\theta})_{\infty}}{(-Cq^{1-s-t+\theta})_\infty} \rphis{2}{1}{q^{-2s}, q^{-2t} }{q^{-2\theta }}{q^2, -\frac{1}{C} q^{1+s+t-\theta}}.
\]
Note that this expression is symmetric in $s$ and $t$, so we can drop the condition $s \leq t$. Using this with $C=C_2$  and collecting terms gives
\[
\begin{split}
\Sigma &=  \left(\frac{\alpha}{\beta} \right)^{s+t} p^{s+t} \hat p^{-t} q^{(s+t)(1+s-t-2m)} q^{-2s} \frac{(\alpha p q^{1+2s-2m-2i\theta -\theta})_{\infty}}{(\alpha p q^{1-2t-2m-2i\theta +\theta})_\infty} \\
& \quad \cdot \rphis{2}{1}{q^{-2s}, q^{-2t} }{q^{-2\theta }}{q^2, \frac{1}{\alpha p} q^{1+2t+2i\theta -\theta+2m}}.
\end{split}
\]
This proves the lemma.
\end{proof}
\vskip.4cm
\noindent
We introduce auxiliary functions again: for $i=1,\ldots,L$,
\[
B_i(y_1,\ldots,y_i) = d_{-1/\beta q}(y_i,x_i;p_i(x),r_i) \hat d_{-q/\beta}(y_i,n_i;\hat p_i(y), r_i) \mu_{-\alpha}(y_i).
\]
Then Lemma \ref{lem2:scalar product 1-site} with
\[
s=x_i,\quad t=n_i, \quad p=p_i(x)=q^{2N_{i-1}^-(x)}, \quad \hat p = \hat p_i(y)=q^{2N_{i-1}^-(y)}, \quad m = N_{i+1}^+(n),
\]
gives the following identity involving the functions $B_i$.

\vskip.2cm
\noindent
\begin{lemma} \label{lem:sum B}
For $i\in \{1,\ldots,L\}$ and $y_1,\ldots,y_{i-1} \in \N$,
\[
\sum_{y_i} B_i(y_1,\ldots,y_i) q^{-2N_{i+1}^+(n)N_{i}^-(y)} = \sum_{y_i} T_{i}(y_i;x,n) q^{-2N_{i}^+(n) N_{i-1}^-(y)} ,
\]
where $T_i(y_i;x,n) = T_i^{(1)}(x,n) T_i^{(2)}(y_i;x,n)$ with
\[
\begin{split}
T_{i}^{(1)}(x,n)= &\left(\frac{\alpha }{\beta} \right)^{x_i+n_i} q^{(x_i+n_i)(1-2N_{i+1}^+(n)-n_i + 2N_{i-1}^-(x)+x_i)} q^{-2x_i}\\
 & \quad \cdot \frac{(\alpha q^{1+N_{i}^-(x)-2N_{i+1}^-(n)-\theta(2i+1)})_{\infty}}{(\alpha q^{1+2N_{i-1}^-(x)-2N_{i}^+(n)-\theta(2i-1)})_\infty}, \\
T_i^{(2)}(y_i;x,n)=& \frac{(q^{-2x_i})_{y_i} (q^{-2n_i})_{y_i}}{(q^2)_{y_i} (q^{-2\theta })_{y_i}} \left(\frac{1}{\alpha} q^{1+2n_i+2N_{i+1}^+(n)-2N_{i-1}^-(x)+2i\theta -\theta}\right)^{y_i} .
\end{split}
\]
\end{lemma}
\vskip.4cm
\noindent
Now we can perform the calculation for $\widetilde D$. We write $\widetilde D(x,n)$ in terms of the auxiliary functions $B_i$,
\[
\begin{split}
\widetilde D(x,n)&=\langle   D^{\text{tr}}_{-1/\beta q}(\cdot , x),\widehat D^{\text{tr}}_{-q/\beta} (\cdot , n)\rangle_{-\alpha}\\
& = \sum_{y_1} B_1(y_1) \sum_{y_2} B_2(y_1,y_2) \ldots \sum_{y_L} B_L(y_1,\ldots,y_L).
\end{split}
\]
Then from Lemma \ref{lem:sum B} and induction we find
\[
\begin{split}
\widetilde D(x,n)&= \sum_y \prod_{i=1}^L T_{i}(y_i;x,n) \\
& = \prod_{i=1}^L T_i^{(1)}(x,n) \sum_{y_i} T_{i}^{(2)}(y_i;x,n) .
\end{split}
\]
Note that
\[
\sum_{y_i} T_{i}^{(2)}(y_i;x,n) = \rphis{2}{1}{q^{-2x_i},q^{-2n_i}}{q^{-2\theta }}{q^2, \frac{1}{\alpha}q^{1+2n_i+2N_{i+1}^+(n)-2N_{i-1}^-(x)+2i\theta -\theta}},
\]
and
\[
\prod_{i=1}^L T_i^{(1)}(x,n) = \left(\frac{\alpha}{\beta} \right)^{N(x+n)}\frac{q^{N(x)(N(x)-1)}}{q^{N(n)(N(n)-1)}}\frac{(\alpha q^{1+2N(x)-2\theta L-\theta})_\infty}{(\alpha  q^{1-2N(n)-\theta})_\infty} ,
\]
where we used that the product of the ratio of the $q-$shifted factorials telescopes, and identities \eqref{N1} (for $n=x$) and \eqref{N2}. So we have
\[
\widetilde D(x,n) = \left(\frac{\alpha}{\beta} \right)^{N(x+n)} \dfrac{q^{N(x)(N(x)-1)}}{q^{N(n)(N(n)-1)}}  \frac{(\alpha q^{1+2N(x)-2\theta L-\theta})_\infty}{(\alpha  q^{1-2N(n)-\theta})_\infty} D(x,n).
\]

\section{Proof for ASIP$(q,\theta)$}\label{ASIP}

In this section we will prove  Theorem \ref{teoasip}. The proof we used for Theorem \ref{teoasep} in the previous section unfortunately does not work for ASIP. The problem lies in the computation of the function $\widetilde D$. To be more precise, the analogue of Lemma \ref{lem:scalar product 1-site} in the ASIP case leads to an infinite sum that, depending on values of $s$, $t$ and $m$, will diverge. However, the computation of the function $D$ for ASIP is completely analogous to the computation for ASEP, and this leads to multivariate $q-$Meixner polynomials as self-duality functions. Because of the similarities between ASIP and ASEP we can make an educated guess for the explicit expression of $\widetilde D$ in terms of $D$, and then verify biorthogonality relations directly.

\vskip.5cm
\noindent
First we need to verify that the function $D$ in Theorem \ref{teoasip} is a self-duality function. We can verify in exactly the same way as for ASEP that
\[
D_\alpha(x,n)=\langle \widehat D^{\text{tr}}_{q/\alpha}(x,\cdot), D^{\text{tr}}_{1/\alpha q}(n,\cdot)\rangle_{-\alpha},
\]
so $D$ is indeed a self-duality function by Theorem \ref{prop:orthogonality}. Note that the function $\widetilde D$ in Theorem \ref{teoasip} is of the form
$C_1(x)C_2(n) D(x,n)$, where $C_1$ and $C_2$ only depend on the total number of particles $N(x)$ and the total number of dual particles $N(n)$. Since the total number of particles is conserved under the dynamics of ASIP, and $D$ is a self-duality function for ASIP, it follows that $\widetilde D$ is also a self-duality function. It only remains to show that $D$ and $\widetilde D$ are biorthogonal with respect to the measure $\mu_\beta$, or equivalently, that functions $D(\, \cdot\,,n)$, $n \in \Lambda_L$, are orthogonal with respect to $C_1 \mu_\beta$.\\

The proof of the orthogonality uses the  orthogonality relations \eqref{ortho} for the $q-$Meixner polynomials $M_n(q^{-x}):=M_n(q^{-x};b,c;q)$
with  $0<b<q^{-1}$ and $c>0$. Using identities for $q-$shifted factorials, these relations can be rewritten as follows:
\begin{equation} \label{orthogonality qMeixner}
\begin{split}
\sum_{x=0}^\infty W(x;b,c;q) M_{n}(q^{-x}) M_{n'}(q^{-x}) =  \delta_{n,n'} H(n;b,c;q),
\end{split}
\end{equation}
with
\[
\begin{split}
W(x;b,c;q)&= \frac{(-bcq^{x+1};q)_\infty (bq;q)_x   }{(q;q)_x} c^x q^{\frac12 x(x-1)},\\
H(n;b,c;q) &=  \frac{(-cq^{-n};q)_\infty (q;q)_{n} }{ (bq;q)_{n} } c^{-n} q^{\frac12 n (n-1)}.
\end{split}
\]

\begin{proposition}
Let $L \in \N$, $c>0$ and $b_i \in (0,q^{-1})$, $i=1,\ldots,L$.  Define multivariate $q-$Meixner polynomials $m_n(x) = m_n(x;b_1,\ldots,b_L,c;q)$ by
\[
m_n(x) = \prod_{i=1}^L M_{n_i}(q^{-x_i};b_i, cB_{i-1} q^{N_{i-1}^-(x)-N_{i+1}^+(n)+i-1};q), \qquad x,n \in \N^L,
\]
where $B_{i} = \prod_{l=1}^{i} b_l$ (the empty product being equal to 1). Moreover, define $w(x) = w(x;b_1,\ldots,b_L,c;q)$ and $h(n)=h(n;b_1,\ldots,b_L,c;q)$ by
\[
\begin{split}
w(x) &= q^{\frac12 N(x)(N(x)-1)} c^{N(x)} (-cB_L q^{N(x)+L};q)_\infty \prod_{i=1}^L \frac{ (b_iq;q)_{x_i} }{(q;q)_{x_i}} (b_iq)^{N_{i+1}^+(x)},\\
h(n) & = q^{\frac12 N(n)(N(n)-1)} c^{-N(n)} (-c q^{-N(n)};q)_\infty \prod_{i=1}^L \frac{(q;q)_{n_i}}{ (b_iq;q)_{n_i} } (b_iq)^{-N_{i+1}^+(n)},
\end{split}
\]
then
\[
\sum_{x \in \N^L} w(x) m_{n}(x) m_{n'}(x) = \delta_{n,n'} h(n).
\]
\end{proposition}
\begin{proof}
We use the shorthand notations
\[
\begin{split}
W_i(x,n) & = W(x_i;b_i, cB_{i-1}q^{N_{i-1}^-(x) - N_{i+1}^+(n) + i-1};q),\\
H_i(x,n) & = H(n_i;b_i, cB_{i-1}q^{N_{i-1}^-(x) - N_{i+1}^+(n) + i-1};q),\\
M_i(x,n) & = M_{n_i}(q^{-x_i};b_i, cB_{i-1}q^{N_{i-1}^-(x) - N_{i+1}^+(n) + i-1};q).
\end{split}
\]
Note that $M_i(x,n)$ and $W_i(x,n)$ depend only on $x_1,\ldots,x_i$ and \emph{not} on $x_{i+1},\ldots,x_L$, and $H_i(x,n)$ depends only on $x_1,\ldots x_{i-1}$ and not on $x_i,\ldots,x_L$. Furthermore, in this notation we have
\[
m_n(x) = \prod_{i=1}^L M_i(x,n).
\]
We have a similar identity involving $w,h,W_i$ and $H_i$: using identities for $N_i^+$, $N_i^-$ and $N$ from Section \ref{functionN} and telescoping products, we obtain
\[
\frac{w(x)}{h(n)} = \prod_{i=1}^L \frac{ W_i(x,n) }{H_i(x,n)}.
\]
Then, for $n,n'\in \N^L$,
\[
\begin{split}
\sum_{x \in \N^L} \frac{w(x) }{h(n)} m_n(x) m_{n'}(x)  &= \sum_{x_1 \in \N} \frac{W_1(x,n)}{H_1(x,n)} M_1(x,n) M_1(x,n') \\
& \qquad \cdot \sum_{x_{2} \in \N} \frac{W_{2}(x,n)}{H_{2}(x,n)} M_2(x,n) M_{2}(x,n') \\
& \qquad  \cdots \sum_{x_L \in \N} \frac{W_L(x,n)}{H_L(x,n)} M_L(x,n) M_L(x,n').
\end{split}
\]
Using the orthogonality relations \eqref{orthogonality qMeixner} for $q-$Meixner polynomials, which imply
\[
\sum_{x_i \in \N} \frac{ W_i(x,n) }{H_i(x,n)} M_i(x,n) M_i(x,n') = \delta_{n_i,n_i'},
\]
we obtain
\[
\sum_{x \in \N^L} \frac{ w(x) }{h(n)} m_n(x) m_{n'}(x) = \delta_{n,n'},
\]
which is the desired orthogonality relation.
\end{proof}
\noindent
The orthogonality relations for the duality functions $D$ and $\widetilde D$ follow from the above orthogonality relations for multivariate $q-$Meixner polynomials by replacing $q$ by $q^2$ and setting
\[
c=\alpha q^{\theta +1}, \quad b_i = q^{2\theta -2}, \quad\text{for } i=1,\ldots,L.
\]

\section{{Orthogonal dualities from symmetries}}\label{Sim}
{In this section we show the link between the self-duality functions constructed in the previous sections and the existence of symmetries of the generator. To do this we rely on the algebraic approach  developed in \cite{CGRS1}-\cite{CGRS} for the construction of the generator  in terms of the Casimir operator of the quantized universal enveloping algebra $\mathcal U_q(\mathfrak{sl}_2)$, where a family of finite, respectively infinite, dimensional representations are used for ASEP($q,\theta$) and ASIP($q,\theta$), respectively. The final aim will be to give an expression in terms of the generators of the algebra for the symmetry $\s_{\alpha,\sigma}$ connected to the orthogonal duality function $\D_{\alpha,\sigma}$.}

\subsection{The quantized enveloping algebra {${\mathcal{U}}_q(\mathfrak{sl}_2)$}}
\label{quantum-algebra}
For $q\in(0,1)$ we consider the complex unital algebra $\mathcal U_q(\mathfrak{sl}_2)$ with generators ${\A^+}, {\A^-}, {{\A}^0}$ satisfying the commutation
relations
\begin{eqnarray}
&& q^{{{\A}^0}} {\A^+} = q \;  {\A^+} q^{{{\A}^0}}, \nn\\
&& q^{{{\A}^0}} {\A^-}= q^{-1}\, {\A^-}q^{{{\A}^0}}, \label{comm-new}\\
&& [{\A^+},{\A^-}]=[2{{\A}^0}]_q. \nonumber
\end{eqnarray}
Here $[A,B]=AB-BA$ is the usual commutator, and
\[
[A]_q:= \frac{ q^A - q^{-A}}{q-q^{-1}}.
\]
(compare to the $q-$number defined in \eqref{[a]_q}).
\noindent
In the limit $q\to 1$ the algebra $\mathcal U_q(\mathfrak{sl}_2)$ reduces
to the enveloping algebra $\mathcal U(\mathfrak{sl}_2)$. The Casimir element $C$ given by
\be
\label{casimir}
C =
\A^+ \A^- + \left[ {\A}^0\right]_q \left[\A^0-1\right]_q
\ee
is in the center of $\mathcal U_q(\mathfrak{sl}_2)$, i.e.~$[C,A]=0$ for all $A  \in \mathcal U_q(\mathfrak{sl}_2)$.

\subsubsection*{Co-product structure}
\label{cooooo}

The co-product for ${\mathcal{U}}_q(\mathfrak{sl}_2)$ is the map
$\Delta: {\mathcal{U}}_q(\mathfrak{sl}_2)\to {\mathcal{U}}_q(\mathfrak{sl}_2) \otimes {\mathcal{U}}_q(\mathfrak{sl}_2)$ given on the generators by
\begin{eqnarray}
\label{co-product2}
\Delta({\A^\pm}) & = & {\A^\pm} \otimes  q^{-{{\A}^0}} + q^{{{\A}^0}} \otimes {\A^\pm}\;, \nonumber \\
\Delta({{\A}^0}) & = & {{\A}^0} \otimes 1 +  1\otimes {{\A}^0}\;,
\end{eqnarray}
and it is extended to $\mathcal U_q(\mathfrak{sl}_2)$ as an algebra homomorphism. In particular $\Delta$ preserves the commutation relations \eqref{comm-new}.

\noindent
We also need iterated coproducts mapping from ${\mathcal{U}}_q(\mathfrak{sl}_2)$ to tensor products of copies of ${\mathcal{U}}_q(\mathfrak{sl}_2)$.
We define iteratively
$\Delta^{n}: {\mathcal{U}}_q(\mathfrak{sl}_2) \to {\mathcal{U}}_q(\mathfrak{sl}_2)^{\otimes (n+1)}$},
i.e. higher powers of $\Delta$, as follows:
\[
\Delta^1:=\Delta\nn, \qquad \Delta^n := (\Delta\otimes \underbrace{1\otimes\ldots\otimes 1}_{n-1 \text{ times}}) \Delta^{n-1}, \quad n \geq 2.
\]
For the generators of $\mathcal U_q(\mathfrak{sl}_2)$ this implies, for $n\ge 2$,
\begin{eqnarray}\label{co-product-L}
\Delta^{n}({\A^\pm}) & = & \Delta^{n-1}({\A^\pm}) \otimes q^{-{{\A}^0}}  +  q^{\Delta^{n-1}({{\A}^0})} \otimes {\A^\pm} \;,  \nonumber\\
\Delta^{n}({{\A}^0}) & = & \Delta^{n-1}({{\A}^0})  \otimes 1 +  \underbrace{1\otimes\ldots\otimes 1}_{n \text{ times}} \otimes {{{\A}^0}}\;.
\end{eqnarray}

\subsubsection*{Representations of the algebra $\mathcal U_q(\mathfrak{sl}_2)$}

From here onward we use the notation $\{|n\rangle \mid  n \in\mathbb K_{\sigma}\}$ for the standard  orthonormal basis of $\ell^2(\mathbb K_{\sigma})$ with $ \mathbb K_{\sigma} =\left\lbrace 0,1, \ldots, \theta\right\rbrace$  if $ \sigma = -1$ and $ \mathbb K_{\sigma} = \mathbb N$ if $ \sigma=1$.
Here and in the following, with abuse of notation, we use the same symbol for a linear operator and the matrix associated to it in a given basis.\\

\noindent
In order to define Markov process generators from the quantized enveloping algebra $\mathcal U_q(\mathfrak{sl}_2)$ we need the following two families of representations.

\paragraph*{Infinite dimensional representations.}
The following ladder operators defined on the standard orthonormal basis of $\ell^2(\N)$ define a family, labeled by $\theta \in \R^+$, of irreducible representations of $\mathcal U_q(\mathfrak{sl}_2)$:
\begin{equation}
\label{stand-repr}
\left\{
\begin{array}{lll}
\A^+ |n\rangle &=& \sqrt{[n+\theta ]_q [n+1]_q}\;| n +1 \rangle
\\
\A^- |n\rangle &=& -\sqrt{[n]_q [n+\theta -1]_q} \;| n-1 \rangle
\\
\A^0 |n\rangle &=& (n+\theta/2) \;| n \rangle \;.
\end{array}
\right.
\end{equation}

\paragraph*{Finite dimensional representations.}
There is a similar representation of $\mathcal U_q(\mathfrak{sl}_2)$ on the finite dimensional Euclidian space $\mathbb C^{\theta+1}$, where $\theta \in \N$.
In this case the irreducible representations of $\mathcal U_q(\mathfrak{sl}_2)$ are labeled by $\theta\in\mathbb{N}$ (corresponding to the dimension of the representation) and given by $(\theta+1)\times(\theta+1)$ dimensional matrices
defined by
\begin{equation}
\label{stand-repr1}
\left\{
\begin{array}{lll}
\A^+ |n\rangle &=& \sqrt{[\theta-n]_q [n+1]_q}\;| n +1 \rangle
\\
\A^- |n\rangle &=& \sqrt{[n]_q [\theta-n+1]_q} \;| n-1 \rangle
\\
\A^0 |n\rangle &=& (n- \theta/2) \;| n \rangle \;.
\end{array}
\right.
\end{equation}

\paragraph*{General case.}
It is possible to collect in a general expression the above defined representations \eqref{stand-repr} and \eqref{stand-repr1}. Recalling the parameter $\si \in \{-1,1\}$ introduced in Section \ref{sigma}, we can write the ladder operators as
\begin{equation}
\label{stand-repr0}
\left\{
\begin{array}{lll}
\A^+ |n\rangle &=& \sqrt{ [\theta+\sigma n]_q [n+1]_q}\;| n +1 \rangle
\\
\A^- |n\rangle &=& -\sigma \sqrt{ [n]_q [\theta +\sigma(n-1)]_q} \;| n-1 \rangle
\\
\A^0 |n\rangle &=& (n+\sigma\theta/2) \;| n \rangle \;.
\end{array}
\right.
\end{equation}
The Casimir element is represented by the diagonal matrix
\[
{C} |n\rangle =[\sigma\theta/2]_q[\sigma\theta/2-1]_q |n\rangle\;.
\]
The adjoints of the operators $A^\pm$ and $A^0$ are given by
\be\label{transp}
(\A^+)^* = -\sigma \A^- \qquad \text{and} \qquad ({{\A}^0})^*={{\A}^0}.
\ee
It is then easily seen that $C^*=C$.

\begin{remark}
The representations we consider are irreducible $*$-representations of two real forms of $\mathcal U_q(\mathfrak{sl}_2)$: for $\si=+1$ we have the discrete series representations of the noncompact real form $\mathcal U_q(\mathfrak{su}(1,1))$, and for $\si=-1$ we have the irreducible representations of the compact real form $\mathcal U_q(\mathfrak{su}(2))$. {Note that for $\si=+1$ we have a representation by unbounded operators. As a dense domain we can take the set of finite linear combinations of basis vectors.}
\end{remark}

\subsection{Construction of  the process from the quantum Hamiltonian}
\label{proc}

\subsubsection*{The quantum Hamiltonian}
\label{q-h}

\vspace{0.2cm}
\noindent
We define the algebraic version of the quantum Hamiltonian $H$ as a sum of coproducts of the Casimir element $C$ given by \eqref{casimir}. The quantum Hamiltonian we are interested in is then the corresponding operator in the representation \eqref{stand-repr0} plus a constant depending on the representation.
\bd[Quantum Hamiltonian]
\label{def-qh}
For $L\in\mathbb{N}$,  $L\ge 2$, the element $H=H^{\phantom x}_{(L)} \in \mathcal U_q(\mathfrak{sl}_2)^{\otimes L}$ is defined by
\be
\label{hami}
H
:=  \sum_{i=1}^{L-1} \bigg\{\underbrace{1\otimes\cdots \otimes 1}_{(i-1) \text{ times}} \otimes \Delta(C) \otimes
\underbrace{1 \otimes \cdots \otimes 1}_{(L-i-1) \text{ times}}\bigg\}\;,
\ee
Then the quantum Hamiltonian $\mathcal H = \mathcal H_{(L)}(\sigma \theta)$ is the operator
\[
\mathcal H = H + c,
\]
where $H$ is the operator in the representation \eqref{stand-repr0} and $c=c_{(L)}(\sigma\theta)$ is a constant uniquely determined by the condition $\mathcal H |0\rangle^{\otimes L}= 0$.

\ed
\vskip.2cm
\noindent
From here on we fix a representation, or equivalently we fix the values of $\si$ and $\theta$, such that $\mathcal H=H+c$. So by $A \in \mathcal U_q(\mathfrak{sl}_2)$ we mean the corresponding operator.
Observe that the quantum Hamiltonian satisfies $\mathcal H^t=\mathcal H$, and that the condition $\mathcal H |0\rangle^{\otimes L} =0$ uniquely determines $c\in \R$, because the state  $|0\rangle \otimes |0\rangle$ is a right eigenvector of $\Delta(C)$.
\noindent
 From \eqref{casimir} and \eqref{co-product2} we have that
\be\label{C}
\begin{split}
\Delta(C) & =
\Delta({\A^+})\Delta({\A^-})+ \Delta([{{\A}^0}]_q) \Delta([{{\A}^0}-1]_q) \\
& =  (q^{{{\A}^0}} \otimes 1)\Bigg \{ {\A^+} \otimes {\A^-}+ {\A^-} \otimes {\A^+}\Bigg \} (1 \otimes q^{-{{\A}^0}})+
 {\A^+} {\A^-} \otimes q^{-2{{\A}^0}} + q^{2{{\A}^0}} \otimes {\A^+} {\A^-}   \\
 & \quad +  \frac{1}{(q-q^{-1})^2}
 \left\{ q^{2{{\A}^0}-1} \otimes q^{2{{\A}^0}} +  q^{1-2{{\A}^0}} \otimes q^{-2{{\A}^0}} -(q+q^{-1}) \right\}.
\end{split}
\ee

One can check that the constant $c$ needed to have $\mathcal H|0\rangle^{\otimes L}=0$  is given by
\be
\label{const}
c={-(L-1)}[\sigma\theta]_q [\sigma\theta-1]_q.
\ee

\noindent
In \cite{CGRS} and \cite{CGRS1} the ASIP($q,\theta$) and ASEP($q,\theta$)  have been constructed from the quantum Hamiltonian via a ground-state transformation. It is possible  to produce a symmetry of the processes by applying the same ground state transformation to a symmetry of the Hamiltonian. The strategy is contained in the following result that has
 been proven in Section 2.1 of \cite{CGRS}.
\bt[Positive ground state transformation]\label{corooo}
Let $\mathcal H$ be a {$|\Omega|\times |\Omega|$} matrix with non-negative off diagonal elements.
Suppose there exists a column vector
$g \in \R^{{|\Omega|}}$
 with strictly positive entries and such that
$\mathcal H g=0$. Let us denote by $G$ the diagonal matrix with entries
{$G(x,x)=g(x)$ for $x\in\Omega$}.
Then we have the following
\begin{itemize}
\item[a)] The matrix
\[
\caL= G^{-1} \mathcal H G
\]
with entries
{
\be\label{amodi}
\caL(x,y) = \frac{\mathcal H(x,y) g(y)}{g(x)}, \qquad x,y \in \Omega\times\Omega
\ee
}
is the generator of a Markov process $\{X_t:t\geq 0\}$ taking values on {$\Omega$}.
\item[b)] $S$ commutes with $\mathcal H$ if and only if $G^{-1} S G$ commutes with $\caL$.
\item[c)] If $\mathcal H=\mathcal H^t$, where $^t$ denotes transposition, then
the probability measure {$\mu$} on {$\Omega$}
{
\be\label{revmes}
\mu(x)= \frac{(g(x))^2}{\sum_{x\in\Omega} (g(x))^2}
\ee
}
is reversible for the process with generator $\caL$.
\end{itemize}
\et

\vskip.2cm
\noindent
The constructive procedure to obtain a suitable ground state matrix $G$ as in Theorem \ref{corooo} is explained in \cite{CGRS} and \cite{CGRS1}. In this paper, as we already know the target processes and corresponding generators ${\cal L}^{\text{\tiny{ASIP}}}$ and ${\cal L}^{\text{\tiny{ASEP}}}$, we restrict ourselves to noticing that, using item c) of Theorem \ref{corooo}, the entries of the ground-state vector $g$ can be written in terms of the reversible measures
${\mu_\alpha^{\text{\tiny{ASIP}}}}$ and ${\mu_\alpha^{\text{\tiny{ASEP}}}}$ given by \eqref{stat-measasip} and \eqref{stat-measasep}.

\subsubsection*{Ground state transformation}
\label{sec5.3}

Let $\mu_\alpha=\mu_{\alpha,\sigma}$, $\alpha\in \R\setminus\{0\}$ be the reversible signed measure defined in \eqref{SM} (in this section we will often omit the dependence on $\sigma$). Then the vectors
\be\label{g}
g_\alpha(x) = \sqrt{\mu_\alpha(x)}
\ee
are  ground states for $\mathcal H$. Notice that, for negative values of $\alpha$, the vector $g_\alpha$ has entries taking values in $\mathbb C$. The diagonal matrix $G_{\alpha}$  represented by a diagonal matrix whose coefficients in the standard basis are given  by \eqref{g}, i.e.
\begin{equation}\label{G}
 G_{\alpha} (x,n)= \sqrt{\mu_\alpha(x)} \cdot \delta_{x,n}\;,
\end{equation}
yields a ground state transformation as in Theorem \ref{corooo} .
For simplicity we denote by $G$ the matrix obtained for the choice $\alpha=1$, $G=G_1$, in which case Theorem \ref{corooo}  applies since the measure $\mu_{1}$ is finite and strictly positive. We have, as a consequence of item a) of Theorem \ref{corooo}, that
the operator ${\cal L}$ conjugated to $\mathcal H$ via $G^{-1}$, i.e.
\be
\label{tildehami}
{\cal L}= G^{-1}\mathcal H G
\ee
is the generator of a Markov jump process $x(t) = (x_1(t),\ldots,x_L(t))$
describing particles jumping on the
chain $\Lambda_L$. In \cite{CGRS} and \cite{CGRS1} it has been proved that the operator $\caL$ is the generator of the ASIP($q,\theta$) and ASEP($q,\theta$), respectively, depending on the choice of $\sigma$. As a consequence of item b) of Theorem \ref{corooo}, if $S$ is a symmetry of $\mathcal H$ (i.e.~$[\mathcal H,S]=0$), then $G^{-1}S G$ is a symmetry of $\caL$.

\vskip.2cm
\noindent
The following proposition, proven in \cite{CGRS1}, allows to construct a duality function for ASIP and ASEP starting from a symmetry of the Hamiltonian.
\bp\label{SD}
If $S$ is a symmetry of $\mathcal H$ then
\begin{itemize}
\item $G^{-1}SG$ is a symmetry for $\mathcal L$,
\item
$D_{1,\alpha}:=G_{\alpha}^{-1}SG_{\alpha}^{-1}$
is a self-duality function for $\mathcal L$,
\item
$D_{2,\alpha}:=G_{\alpha}^{-1}(S^t)^{-1}G_{\alpha}^{-1}$
is a self-duality function for $\mathcal L$,
\item $D_{1,\alpha}$ and $D_{2,\alpha}$ are orthogonal with respect to the measure $G^2_\alpha(x)$, i.e. $D_{1,\alpha}G^2_\alpha D_{2,\alpha}^t=G^{-2}_\alpha$.
\end{itemize}
\ep

\subsubsection*{Symmetries}

At this aim we  need a non-trivial symmetry which yields a non-trivial ground state.
Starting from the basic symmetries of $H$  and inspired by the analysis of the symmetric case ($q\rightarrow 1$),
it will be convenient to consider the {\em exponential} of those symmetries.

\subsection{Symmetries associated to the self-duality functions} \label{symm}

We use the following $q-$exponential functions:
\beq
&& E_{q^2}(z):=(-z)_\infty=\sum_{n=0}^\infty q^{n(n-1)}\;\frac{z^n}{(q^2)_n}, \qquad
 e_{q^2}(z):=\frac{1}{(z)_\infty}=\sum_{n=0}^\infty \frac{z^n}{(q^2)_n}\quad \text{for }\:|z|<1.\nn
\eeq
These satisfy $e_{q^2}(z)E_{q^2}(-z)=1$, and the following factorization rules: if $x$ and $y$ satisfy $xy = q^2yx$, then
\beq\label{fact}
&&  E_{q^2}(x+y)=E_{q^2}(x)\cdot E_{q^2}(y)\quad \text{and} \quad e_{q^2}(x+y)=e_{q^2}(y)\cdot e_{q^2}(x).
\eeq
With the $q-$exponential functions we define the following operators: for $\alpha>0$
\[
\begin{split}
S^{\mathrm{tr}}_\alpha &:= e_{q^2}\left(\sqrt \alpha(1-q^2) \cdot \Delta^{L-1}(q^{{{\A}^0}}{\A^+})\right),\\
 \widehat S^{\mathrm{tr}}_\alpha &:=E_{q^2}\left(\sqrt \alpha q^{\tfrac{\sigma\theta}2(1+2L)}(1-q^2)\Delta^{L-1}(q^{-{{\A}^0}}{\A^+})\right).
\end{split}
\]
In case we work in an infinite dimensional representation, i.e.~$\sigma=+1$, we should be careful with convergence of the series obtained from applying these operators to functions. If we apply these operators only to finitely supported functions there are no convergence issues.
We have the following lemma.
\bl\label{Sym}
For all $\alpha>0$, $\widehat S^{\rm{tr}}_\alpha$ and $ S^{\rm{tr}}_\alpha$ are symmetries of $\mathcal H$, i.e.
\beq
[\mathcal H,  S^{\rm{tr}}_\alpha]=0, \qquad  [\mathcal H,  \widehat S^{\rm{tr}}_\alpha]=0.
\eeq
\el
\bpr
This follows from the fact that $\A^+$ and $\A^0$ commute with the Casimir operator $C$, and then $\Delta^{L-1}(q^{\pm{{\A}^0}}{\A^+})$ commutes with $1\otimes \ldots \otimes 1 \otimes \Delta(C) \otimes 1 \otimes \ldots \otimes 1$. See Section 4 of  \cite{CGRS} for more details.
\epr
\subsection*{Triangular dualities}

In the spirit of Section 4 of \cite{CGRS}, the following proposition shows that we can write the triangular dualities in terms of the symmetries $S^{\mathrm{tr}}_\alpha$ and $\widehat S^{\mathrm{tr}}_\alpha$ given in Lemma \ref{Sym}.
We first define two diagonal matrices by
\beq\label{diag}
&&A(x,n):= q^{N^2(x)} \, \delta_{x,n}, \qquad B(x,n):= q^{N(x)} \, \delta_{x,n}
\eeq
\bp\label{TrDu}
Let $D_\lambda^{\mathrm{tr}}$ and $\widehat D_\lambda^{\mathrm{tr}}$ be the triangular self-duality functions defined in \eqref{tria1asip}, then we have:
\beq\label{D1}
 D_{1/\alpha q}^{\rm{tr}}=B^{-1}G^{-1}_\alpha  S^{\rm{tr}}_{\alpha} G^{-1}_\alpha A
 \eeq
and
\beq\label{D2}
  \widehat D_{q/\alpha}^{\rm{tr}}=BG^{-1}_\alpha \widehat S^{\rm{tr}}_{\alpha} G^{-1}_\alpha A^{-1}.
\eeq
\ep
\noindent
Here we consider a duality function $D$ as  the matrix with elements $D(x,n)$, while we denote $D^{t}$ the transpose matrix. The proof of Proposition \ref{TrDu} is given in Section \ref{proof of TrDu}.

\subsection*{Orthogonal dualities}
Now we fix $\alpha>0$ and  use Proposition \ref{TrDu} and the expression \eqref{qua}  to write the orthogonal dualities
and the associated symmetries in terms of  the symmetries $S^{\mathrm{tr}}_\alpha$ and $\widehat S^{\mathrm{tr}}_\alpha$. We first define the following diagonal operators:
\[
\begin{split}
	M(x,n)&:= (-1)^{N(x)} \, \delta_{x,n},\\
	R_\alpha(x,n)&:= (-\sigma \alpha q^{1+2N(x) +\sigma\theta(2 L+1)})_\infty\, \delta_{x,n},\\
	T_\alpha(x,n)&:= (-\sigma\alpha q^{1-2N(x)+\sigma\theta})_\infty  \, \delta_{x,n}.
\end{split}
\]
Let $\D_{\alpha,\sigma}$  be the normalized orthogonal self-duality function defined  in equation \eqref{normalizedselfduality} and $\s_{\alpha,\sigma}$ its associated symmetry \eqref{S} then we have
\be\label{dd}
\D_\alpha= G_\alpha^{-1} (ABR_\alpha)^{\tfrac 12} \cdot \widehat S^{\mathrm{tr}}_\alpha M  (S^{\mathrm{tr}}_\alpha)^t \cdot (ABT_\alpha)^{-\tfrac 1 2} G_\alpha^{-1}
\ee
and
\be\label{ss}
\s_\alpha= G_\alpha^{-1} (ABR_\alpha)^{\tfrac 12} \cdot \widehat S^{\mathrm{tr}}_\alpha M  (S^{\mathrm{tr}}_\alpha)^t \cdot (ABT_\alpha)^{-\tfrac 1 2} G_\alpha.
\ee
\bpr
From \eqref{qua} we have that $D_\alpha$ can be given in terms of scalar products of the triangular dualities. In matrix form this reads
\be
D_\alpha=\widehat D^{\rm{tr}}_{q/\alpha} G^{2}_{-\alpha} (D^{\rm{tr}}_{1/\alpha q})^t,
\ee
then, using
 the expressions in Proposition \ref{TrDu}, it follows that
 \be\label{OrtDu}
D_\alpha= B G^{-1}_\alpha \widehat S^{\rm{tr}}_{\alpha}M \(S^{\rm{tr}}_{\alpha}\)^t G^{-1}_\alpha B^{-1}.
\ee
Then \eqref{dd} follows from
\be \label{OrtDu2}
\D_\alpha=(AB^{-1}R_\alpha)^{\tfrac 12 } D_\alpha (A^{-1}BT_\alpha^{-1})^{\tfrac 1 2}
\ee
and \eqref{ss} follows from \eqref{dd} and the fact that
\be
\s_{\alpha}= \D_{\alpha}G^2_\alpha.
\ee
This concludes the proof.
\epr

\br
Notice that we can rewrite the orthogonality relation \eqref{Ort} of $\D_\alpha$ as
\be
\D_\alpha^t G_\alpha^2\D_\alpha=G_\alpha^{-2}
\ee
and the unitarity property of $\s_\alpha$
as follows:
\be
\s_\alpha^t G_\alpha^2 \s_\alpha=G_\alpha^2.
\ee
{These identities imply relations between $q$-exponentials of generators of $\mathcal U_q(\mathfrak{sl}_2)$. Such relations have been exploited in e.g.~\cite{KMM},\cite{GV} to obtain orthogonality relations for specific $q$-hypergeometric functions.}
\er

\br
 In the infinite dimensional setting, $\si =  +1$, this should be interpreted as a formal identity; as this is an identity involving unbounded operators, the above calculation is not all rigorous.

\er

\subsection{Proof of Proposition \ref{TrDu}.} \label{proof of TrDu}
We first compute the action of the symmetries associated to the triangular dualities.
\subsubsection*{Action of   $ S^{\mathrm{tr}}_\alpha$.}
We have
\beq
  S^{\mathrm{tr}}_\alpha= e_{q^2}(\sqrt \alpha  (1-q^2) \cdot \Delta^{L-1}(q^{{{\A}^0}}{\A^+})),\nn
\eeq
where
\beq
&& \Delta^{L-1}(q^{+{{\A}^0}}{\A^+})= q^{{{\A}}_1^0} {\A}_1^+ + q^{2{{\A}}_1^0 +{{\A}}_2^0} {\A}_2^+ + ...+  q^{2 \sum_{i=1}^{L-1} {{\A}}_i^0 + {{\A}}_L^0} {\A}_L^+. \nn
\eeq
From \eqref{comm-new} we know that
\beq
q^{2{{\A}^0}} q^{{{\A}^0}}{\A^+} =q^2  q^{{{\A}^0}}{\A^+} q^{2{{\A}^0}},
\eeq
then from \eqref{fact} we have
\beq
 S^{\mathrm{tr}}_\alpha= S_1^+  S_2^+ \ldots  S_L^+\nn
\eeq
with
\beq
 S^+_i=e_{q^2}(\sqrt \alpha (1-q^2) q^{2\sum_{m=1}^{i-1}\A_m^0 +\A_i^0}\A_i^+).\nn
\eeq
Then, for $\sigma=1$,
\beq
 S^{\mathrm{tr}}_\alpha|n\rangle=\sum_{\ell_1,\ldots,\ell_L}\prod_i \sqrt{\binom{n_i+\ell_i}{\ell_i}_q\cdot \binom{n_i+\ell_i+\theta -1}{\ell_i}_q} \cdot q^{\ell_i(n_i+\theta/2+1)+2\ell_iN^-_{i-1}(n+ \theta/2)} \alpha^{\tfrac{\ell_i}2}  |n+\ell\rangle,\nn
\eeq
so that
\beq
 S^{\mathrm{tr}}_\alpha(x,n)=\prod_i \sqrt{\binom{x_i}{n_i}_q\cdot \binom{x_i+\theta -1}{n_i+\theta -1}_q} \cdot q^{(x_i-n_i)[(n_i+\theta/2+1)+2N^-_{i-1}(n+\theta/2)]}\alpha^{\tfrac{x_i-n_i}2} \cdot \mathbf 1_{x_i\ge n_i}.\nn
\eeq
For $\sigma=-1$,
\beq
 S^{\mathrm{tr}}_\alpha(x,n)=\prod_i \sqrt{\binom{x_i}{n_i}_q\cdot \binom{\theta- n_i}{\theta- x_i}_q} \cdot  q^{(x_i-n_i)[(n_i-\theta/2+1)+2N^-_{i-1}(n-\theta/2)]} \alpha^{\tfrac{x_i-n_i}2} \cdot \mathbf 1_{x_i\ge n_i}.\nn
\eeq

\subsubsection*{Action of $\widehat S^{\mathrm{tr}}_\alpha$.} We have
\beq
\widehat S^{\mathrm{tr}}_\alpha:=E_{q^2}(\sqrt \alpha q^{\tfrac{\sigma\theta}2(1+2L)}(1-q^2)\Delta^{L-1}(q^{-{{\A}^0}}{\A^+})),\nn
\eeq
where
\beq
&& \Delta^{L-1}(q^{-{{\A}^0}}{\A^+})= q^{-{{\A}}_1^0} {\A}_1^+ + q^{-2{{\A}}_1^0 -{{\A}}_2^0} {\A}_2^+ + ...+  q^{-2 \sum_{i=1}^{L-1} {{\A}}_i^0 - {{\A}}_L^0} {\A}_L^+. \nn
\eeq
From \eqref{comm-new} we know that
\beq
q^{-2{{\A}^0}} q^{-{{\A}^0}}{\A^+} =q^{-2}  q^{-{{\A}^0}}{\A^+} q^{-2{{\A}^0}},
\eeq
then, from \eqref{fact} we have
\beq
 \widehat S^{\mathrm{tr}}_\alpha= \widehat S_L^+  \widehat S_{L-1}^+ \ldots  \widehat S_1^+\nn,
\eeq
with
\beq
 \widehat S^+_i=E_{q^2}(\sqrt \alpha q^{\tfrac{\sigma\theta}2(1+2L)}(1-q^2) q^{-2\sum_{m=1}^{i-1}\A^0_m -\A^0_i}\A^+_i).\nn
\eeq
Then it follows that
\beq
\widehat S^{\mathrm{tr}}_\alpha(x,n) =\prod_i \sqrt{\binom{x_i}{n_i}_q \cdot\binom{x_i+\theta -1}{n_i+\theta -1}_q} \alpha^{\tfrac{x_i-n_i}2} \cdot
q^{-(x_i-n_i)[2N^+_{i+1}(n+\theta/2)+(n_i-\theta L+1)]} \cdot \mathbf 1_{x_i\ge n_i} \nn
\eeq
for $\sigma=+1$ and
\beq
\widehat S^{\mathrm{tr}}_\alpha(x,n) =\prod_i \sqrt{\binom{x_i}{n_i}_q \cdot\binom{\theta- n_i }{\theta- x_i}_q} \alpha^{\tfrac{x_i-n_i}2}\cdot
q^{-(x_i-n_i)[2N^+_{i+1}(n-\theta/2)+(n_i+\theta L+1)]} \cdot \mathbf 1_{x_i\ge n_i} \nn
\eeq
for $\sigma=-1$.

\vskip.4cm
\noindent
To complete the proof we will make  use of the following Lemma:
\bl\label{L:formula}
 For $n\ge m$,
\be\label{formula}
\sqrt{\frac{\binom{n}{m}_q\cdot \binom{n+\theta -1}{m+\theta -1}_q}{\binom{m+\theta -1}{m}_q\cdot \binom{n+\theta -1}{n}_q}}=\frac{\binom{n}{m}_q}{\binom{m+\theta -1}{m}_q}
\qquad
\text{and}
\qquad
\sqrt{\frac{\binom{n}{m}_q\cdot \binom{\theta-m}{\theta-n}_q}{\binom{\theta}{m}_q\cdot \binom{\theta}{n}_q}}=\frac{\binom{n}{m}_q}{\binom{\theta}{m}_q}.
\ee
\el
\vskip.4cm
\noindent
Now we can conclude the proof of  Proposition \ref{TrDu}.
\subsubsection*{Proof of \eqref{D1}}
Using \eqref{formula} and \eqref{G} we find that the corresponding triangular duality is given by
\[
\begin{split}
G^{-1}_\alpha S^{\mathrm{tr}}_\alpha G^{-1}_\alpha(x,n)  =\prod_i &\frac{\binom{x_i}{n_i}_q}{\Psi_{q,\sigma}(\theta, n_i)} \;q^{(x_i-n_i)[(n_i+\theta/2+1)+2N^-_{i-1}(n+\theta/2)]-\theta  i(n_i+x_i)} \cdot  \, \alpha^{- n_i}.
\end{split}
\]
Now,
using that
\beq
\sum_i(x_i-n_i)[(n_i+\theta/2)+2N^-_{i-1}(n+\theta/2)-\theta i]=-N^2(n)+\sum_i x_i(n_i+2N_{i-1}^-(n))
\eeq
we get
\[
\begin{split}
G^{-1}_\alpha  S^{\mathrm{tr}}_\alpha G^{-1}_\alpha(x,n)& =q^{-N^2(n)+N(x)}  \cdot \prod_i \frac{\binom{x_i}{n_i}_q}{\Psi_{q,\sigma}(\theta, n_i)} \cdot q^{x_i(n_i+2N^-_{i-1}(n))}q^{-2\theta  in_i}\cdot \(\frac 1 {\alpha q}\)^{n_i}.
\end{split}
\]
Comparing this with \eqref{tria1asip} we obtain
\beq
G^{-1}_\alpha  S^{\mathrm{tr}}_\alpha G^{-1}_\alpha(x,n)=q^{-N^2(n)+N(x)} \cdot D_{1/\alpha q}^{\mathrm{tr}}(x, n )\nn
\eeq
from which the statement follows.

\subsubsection*{Proof of \eqref{D2}.}
Using  \eqref{formula} and \eqref{G} we obtain
\[
\begin{split}
G^{-1}_\alpha \widehat S^{\mathrm{tr}}_\alpha G^{-1}_\alpha(x,n) =\prod_i & \frac{\binom{x_i}{n_i}_q}{\Psi_{q,\sigma}(\theta, n_i)}
q^{-(x_i-n_i)[2N^+_{i+1}(n+\theta/2)+(n_i-\theta L+1)]-\theta  i(x_i+n_i)}  \alpha^{- {n_i}}.
\end{split}
\]
We use that
\[
\begin{split}
\sum_i(x_i-n_i)[2N^+_{i+1}(n+\theta/2) & +(n_i+\theta/2)+\theta i]= \\
& -N^2(n)+\theta/2(1+2L)N(x-n)+\sum_i n_i(2N^-_{i-1}(x)+x_i)
\end{split}
\]
to get
\[
\begin{split}
G^{-1}_\alpha  \widehat S^{\mathrm{tr}}_\alpha G^{-1}_\alpha(x,n) & = q^{N^2(n)-N(x)}  \cdot\prod_i\frac{\binom{x_i}{n_i}_q}{\Psi_{q,\sigma}(\theta, n_i)}\cdot
q^{-n_i[2N^-_{i-1}(x)+x_i]-2\theta  in_i} \cdot (\tfrac q \alpha)^{n_i} .
\end{split}
\]
Then comparing with \eqref{tria2asip} we obtain
\beq
G^{-1}_\alpha \widehat S^{\mathrm{tr}}_\alpha G^{-1}_\alpha(x,n) = q^{N^2(n)-N(x)}  \cdot  \widehat D_{q/\alpha}^{\mathrm{tr}}( x, n )\;,\nn
\eeq
which is the desired result. \qed

\section{Appendix} \label{appendix}
See Section \ref{ssec:notation} for the definition of the $q-$binomial coefficients and $q-$Pochhammer symbols. We refer to Appendix I of \cite{gasper2004basic} for the formulas involving $q-$Pochhammer symbols. Identity \eqref{qbinid} follows directly from the definition of the $q-$binomial coefficient, and \eqref{newton} is (a special case of) the $q-$binomial formula \cite[(II.4)]{gasper2004basic}.
\subsection{Identities for $q-$binomial coefficients }

 For $n,x,y \in \N$,
\be \label{qbinid}
{\binom{y}{n}_q \binom{x}{y}_q}= {\binom{x}{n}_q \binom{x - n}{y - n}_q},
\ee
moreover,   for $N \in \N$ and $t \in \R$,
\be \label{newton}
\sum_{\kappa =0}^N \binom{N}{\kappa}_q \, q^{\kappa N}  \(tq^{-1}\)^\kappa=\prod_{\kappa=1}^N (1+t q^{2(\kappa -1)})=(-t)_N.
\ee
\subsection{Identities for $q-$Pochhammer symbols}
 For $n, m \in \N$ and $a \neq 0$ we have
\begin{equation} \label{sum}
  (a)_{n+m} = (a)_m (aq^{2m})_n
\end{equation}
\begin{equation} \label{diff}
 (a)_{m-n} = \frac{ (a)_m }{ (q^{2-2m}/a)_n } \left( -\frac{q^2}{a} \right)^n q^{n(n-1) - 2mn} \;,
\end{equation}
moreover,  for $b\neq 0$, $c\neq 0$,
\begin{equation} \label{1.8.15koeoek}
\dfrac{(bq^{-2n})_{n}}{(cq^{-2n})_{n}} = \left( \dfrac{b}{c} \right)^{n} \dfrac{(b^{-1}q^{2})_{n}}{(c^{-1}q^{2})_{n} } \;,
\end{equation}
 finally, for $n,m \in \N$, $n \geq m$,
\begin{equation} \label{eq:id Pochhammer}
\frac{(q^2)_n}{(q^2)_{n-m} (q^{-2n})_m } = (-1)^m q^{2mn- m(m-1)}.
\end{equation}
\subsection{Identities for $q-$hypergeometric functions}
We refer to the book \cite{gasper2004basic} for theory on $q-$hypergeometric functions. Here we only use the $q-$hypergeometric function
\be
\label{eq:2phi1}
\rphis{2}{1}{a, b}{c}{q,z} := \sum_{k=0}^{\infty}  \dfrac{(a;q)_{k} (b;q)_{k}  }{(c;q)_{k}  }
\dfrac{z^k}{(q;q)_{k}} \;.
\ee
where, as before, $(a;q)_k = \prod_{i=0}^{k-1} (1-aq^{i})$. We always assume that $c \not\in q^{-\N}$, so that the denominator never equals zero. The series converges absolutely for $|z|<1$. Note that for $a=q^{-n}$, $n \in \N$, the series terminates after the $(n+1)$-th term; in this case the series is a polynomial of degree $n$ in $b$.

\noindent
The $_2\varphi_1$-functions we encounter in this paper will depend on $q^2$ instead of $q$. We need the following two transformation formulas for $_2\varphi_1$-functions. The first is one of Heine's transformation formulas, see \cite[(III.3)]{gasper2004basic}, which is valid as long as the series on both sides converge. The second one is only valid for a terminating $_2\varphi_1$-series, and is obtained from reversing the order of summation.\\

\noindent
Heine's transformation:
\begin{equation} \label{euler}
\rphis{2}{1}{a, b}{c}{q^{2},z} = \frac{ (\frac{abz}{c})_\infty }{(z)_\infty} \rphis{2}{1}{c/a, c/b}{c}{q^{2},\frac{abz}{c}}.
\end{equation}
Transformation for terminating series:
\beq\label{change}
&&\rphis{2}{1}{q^{-2n}, b}{c}{q^{2},z} = \frac{  (b)_n }{(c)_n }q^{-n-n^2} (-z)^n\rphis{2}{1}{q^{-2n}, q^{2-2n}c^{-1}}{q^{2-2n}b^{-1}}{q^{2}, \frac{cq^{2+2n}}{bz}}.\nn\\
\eeq

\begin{lemma} \label{applemm} For $n,x,y\in \N$, $x\le n$, $m\in \R^+$ and $|C|<q^{x+m+n-1}$ we have
\begin{equation*}
\begin{split}
& \sum_{y =n }^\infty  \binom{y}{x}_q \binom{y}{n}_q \binom{y+m-1}{y}_q C^y\\
 & = \frac{ C^n q^{n(1-m-x)}(q^{2m})_n  {q^{x^2}}}{(q^2)_x (q^2)_{n-x} } \rphis{2}{1}{q^{2(n+1)}, q^{2(m+n)}}{q^{2(1+n-x)}}{q^2, Cq^{1-x-n-m}}.
\end{split}
\end{equation*}
\end{lemma}
\bpr
By some algebraic manipulation of the $q-$numbers and changing the variable in the summation we get
\begin{equation*}
\begin{split}
& \sum_{y =n }^\infty  \binom{y}{x}_q \binom{y}{n}_q \binom{y+m-1}{y}_q \; C^y\\
& = \frac{q^{x^2+n^2}}{(q^2)_x (q^2)_n}\sum_{y=n}^\infty \frac{ (q^2)_y (q^{2m})_y }{(q^2)_{y-x} (q^2)_{y-n} }C^y q^{-y(x+n+m-1)} \\
& = \frac{ C^n q^{n(1-m-x-n)} {q^{x^2+n^2}}}{(q^2)_x (q^2)_n } \sum_{r=0}^\infty \frac{ (q^2)_{r+n} (q^{2m})_{r+n}}{(q^2)_{r+n-x} (q^2)_r } \;C^r q^{r(1-m-x-n)}\\
& = \frac{ C^n q^{n(1-m-x)}(q^{2m})_n  {q^{x^2}}}{(q^2)_x (q^2)_{n-x} } \sum_{r=0}^\infty \frac{ (q^{2(n+1)})_{r} (q^{2(m+n)})_{r}}{(q^{2(1+n-x)})_{r} (q^2)_r } \;C^r q^{r(1-m-x-n)} \\
& = \frac{ C^n q^{n(1-m-x)}(q^{2m})_n  {q^{x^2}}}{(q^2)_x (q^2)_{n-x} } \rphis{2}{1}{q^{2(n+1)}, q^{2(m+n)}}{q^{2(1+n-x)}}{q^2, Cq^{1-x-n-m}}.
\end{split}
\end{equation*}
That concludes the proof.
\epr
\begin{lemma} \label{applemma}
For $x,n,m\in \N$ with $x\le n \leq m$ and $C \in \R$ we have
\begin{equation*}
\begin{split}
& \sum_{y =n }^{{m}}  \binom{y}{x}_q \binom{y}{n}_q \binom{m}{y}_q C^y\\
 & = \frac{ (-C)^n q^{n(1+m-x)}(q^{-2m})_n  {q^{x^2}}}{(q^2)_x (q^2)_{n-x} } \rphis{2}{1}{q^{2(n+1)}, q^{-2(m-n)}}{q^{2(1+n-x)}}{q^2, -Cq^{1+m -x-n}}.
\end{split}
\end{equation*}
\end{lemma}
\noindent
We omit the proof  of this identity which is similar to that of Lemma \ref{applemm}.
\subsection{$q-$Orthogonal polynomials} \label{app9.4}

\subsubsection*{$q-$Krawtchouk polynomials}\label{Kraw}
The $q-$Krawtchouk polynomials  in the $q$-hypergeometric representation are given by:
\begin{equation} \label{def:q-Krawtchouk}
K_n(q^{-x};p,c;q):=  \rphis{2}{1}{q^{-x},q^{-n}}{q^{-c}}{q,p{q^{n+1}}}, \qquad \text{for} \quad c \in \N, \quad n,x \in \lbrace 0,\ldots, c \rbrace
\end{equation}
and  $_{2}\phi_{1}$ as in definition \eqref{eq:2phi1}. We remark that in the literature \cite{Koekoek} they are known as quantum $q-$Krawtchouk polynomials.
\vskip.2cm
\noindent
{\bf Orthogonality relations.}
 Under the condition
 \be\label{condK}
 pq^{c}>1\qquad \text{and} \qquad c\in \N
 \ee
   these polynomials are orthogonal with respect to a positive measure on $\{0, 1, \ldots, c\}$, see \cite[\S14.14]{Koekoek}.
The orthogonality relations for $q-$Krawtchouk polynomials read as follows:
\beq
&&\sum_{x=0}^c  \;  \frac{ (pq;q)_{c-x}
\;(-1)^{c-x}}{(q;q)_x (q;q)_{c-x}}\;  \; q^{\binom{x}{2}}\cdot K_m(q^{-x};p,c;q)\cdot   K_n(q^{-x};p,c;q)\nn\\
&&=\frac{ (-1)^n \, p^c\, (q;q)_{c-n} (q;q)_n(pq;q)_n}{((q;q)_c)^2 } \; \cdot  q^{\binom{c+1}{2}-\binom{n+1}{2}+cn}   \cdot \delta_{m,n}.\label{orthoK}
\eeq

\subsubsection*{$q-$Meixner polynomials}\label{Meix}

The  $q-$Meixner polynomials in the $q$-hypergeometric representation are given by
\beq \label{q-Meixner}
M_n(q^{-x};b,c;q):=  \rphis{2}{1}{q^{-x},q^{-n}}{bq}{q,-\frac{q^{n+1}}{c}}, \qquad \text{for} \quad x,n \in \N,
\eeq
where $_{2}\phi_{1}$ is the  $q-$hypergeometric function defined in \eqref{eq:2phi1}. Note that $M_n(q^{-x};b,c;q)$ is a polynomial in $q^{-x}$ of degree $n$, but it is also a polynomial in $c^{-1}$ of degree $n$.
We remark the similarity with the $q-$Krawtchouk polynomials: for $c \in \N$ we have
 $K_n(q^{-x};p,c;q)=M_n(q^{-x};q^{-1-c},-p^{-1};q)$. 

\vskip.2cm
\noindent
{\bf Orthogonality relations.}
 Under the conditions
 \be\label{condM}
 bq\in [0,1) \qquad \text{and} \qquad c>0
 \ee
  these polynomials are orthogonal with respect to a positive measure on $\N$, see \cite[\S14.13]{Koekoek}.
The orthogonality relations for $q-$Meixner polynomials read as follows:
\beq
&&\sum_{x=0}^\infty  \;  \frac{ (bq;q)_x
\;c^x}{(q;q)_x (-cbq;q)_x}\;  \; q^{\binom{x}{2}}\cdot M_m(q^{-x};b,c;q)\cdot   M_n(q^{-x};b,c;q)\nn \\
&&=\frac{ (-c;q)_\infty (q;q)_n(-c^{-1}q;q)_n}{(-cbq;q)_\infty (bq;q)_n} \; \cdot  q^{-n}   \cdot \delta_{m,n}.\label{ortho}
\eeq
The function $M_n(q^{-x};b,c;q)$, $x \in \N$, is also a polynomial in $q^n$ of degree $x$. It can be considered as an instance of a rescaled big $q-$Laguerre polynomial, see \cite[\S14.11]{Koekoek},
\[
M_n(q^{-x};b,c;q) = (-q^{-x}/bc;q)_x P_n(bq^{1+n};b,-bc;q).
\]
The big $q-$Laguerre polynomials $P_m(y;\alpha,\beta;q)$ with $0<\alpha q<1$ and $\beta<0$ satisfy orthogonality relations of the form
\[
\begin{split}
\sum_{k=0}^\infty &P_m(\beta q^{1+k};\alpha,\beta;q) P_n(\beta q^{1+k};\alpha,\beta;q) w(\beta q^{k+1}) \\ & + \sum_{k=0}^\infty P_m(\alpha q^{1+k};\alpha,\beta;q) P_n(\alpha q^{1+k};\alpha,\beta;q) w(\alpha q^{k+1}) = \delta_{m,n} N_n,
\end{split}
\]
where the weight function $w$ and the squared norm $N_n$ are known explicitly. We see that the big $q$-Laguerre polynomials are orthogonal with respect to a measure supported on the finite interval $[\beta q,\alpha q]$, hence they form a complete orthogonal basis for the corresponding weighted $L^2$-space. It follows that they also satisfy the dual orthogonality relation
\[
\sum_{n=0}^\infty P_n(y;\alpha,\beta;q) P_n(y';\alpha,\beta;q) N_n^{-1} = \frac{ \delta_{y,y'}}{w(y)}, \qquad y,y'\in \beta q^{1+\N} \cup \alpha q^{1+\N}, \label{dual orth rel}
\]
and the set $\{n \mapsto P_n(y;\alpha,\beta;q) \mid y \in  \beta q^{1+\N} \cup \alpha q^{1+\N} \}$ is a complete orthogonal basis for the weighted $L^2$-space $\ell^2(\N;N_n^{-1})$. The orthogonality relations for the $q$-Meixner polynomials are equivalent to the dual orthogonality relations of the big $q-$Laguerre polynomials for $y \in \alpha q^{1+\N}$, so the $q$-Meixner polynomials do not form a complete basis for their weighted $L^2$-space.

\vspace{1 cm} \noindent
\textbf{Acknowledgment.} The authors thank Cristian Giardin\`{a} and Frank Redig for useful discussion. C.F. acknowledges support from the European Research Council (ERC) under the European Union’s Horizon 2020 research and innovative program (grant agreement No 715734). We also thank the referees for their helpful comments which helped improving the quality of the manuscript.

\end{document}